\documentclass[dvips,preprint,11pt]{imsart}

\usepackage{amsmath}
\usepackage{amsfonts}
\usepackage{mathrsfs}
\usepackage{graphicx}
\usepackage{epstopdf}
\usepackage{ amsmath, amsfonts, amssymb, amsthm, amscd}
\usepackage[T1]{fontenc}
\usepackage[english]{babel}
\usepackage{color}
\usepackage{pgf,pgfarrows,pgfnodes,pgfautomata,pgfheaps}
\usepackage{amsmath,amssymb}
\usepackage[latin1]{inputenc}

\usepackage[latin1]{inputenc}
\usepackage{graphicx}

\fontsize{11pt}{14.5pt}\selectfont

\def\Dy#1{\Frac{\partial #1}{\partial y}}

\def\Dy_1y_1#1{\Frac{\partial^2 #1}{\partial y_1^2}}

\def\reff#1{{\rm(\ref{#1})}}

\newtheorem{Theorem}{Theorem}[part]
\newtheorem{Definition}{Definition}[part]
\newtheorem{Proposition}{Proposition}[part]

\newtheorem{Lemma}{Lemma}[part]

\newtheorem{Remark}{Remark}[part]

\makeatletter \@addtoreset{equation}{section}

\@addtoreset{Definition}{section}

\@addtoreset{Theorem}{section}

\@addtoreset{Proposition}{section}

\@addtoreset{Assumption}{section}

\@addtoreset{Corollary}{section}

\@addtoreset{Lemma}{section}

\@addtoreset{Remark}{section}

\@addtoreset{Example}{section}

\makeatother \makeatletter

\def \Int{\displaystyle\int}
\def \Frac{\displaystyle\frac}
\def \Inf{\displaystyle\inf}
\def \Sup{\displaystyle\sup}

\def \Max{\displaystyle\max}

\def \be{\begin{eqnarray}}
\def \ee{\end{eqnarray}}
\def \b*{\begin{eqnarray*}}
\def \e*{\end{eqnarray*}}

\def \E{\mathbb{E}}

\def \L{\mathbb{L}}
\def \N{\mathbb{N}}
\def \P{\mathbb{P}}
\def \R{\mathbb{R}}

\def \W{\overleftarrow{W}}
\def \B{\overleftarrow{B}}
\def \[{[\,\!\![}
\def \]{]\,\!\!]}

\def \1{{\bf 1}}

\def \ep{\hbox{ }\hfill$\Box$}

\def\reff#1{{\rm(\ref{#1})}}

\def\Cc{{\cal C}}
\def\Dc{{\cal D}}

\def\Fc{{\cal F}}
\def\Gc{{\cal G}}
\def\Hc{{\cal H}}

\def\Lc{{\cal L}}

\def\Nc{{\cal N}}

\def\Lc{{\cal L}}

\def\Sc{{\cal S}}

\def\Oc{{\cal O}}

\def\Rc{{\cal R}}

\begin{document}

\begin{frontmatter}

\title{Numerical Computation for Backward Doubly SDEs with random terminal time}
\runtitle{Numerical computation  for BDSDEs}


\author{\fnms{Anis} \snm{Matoussi}\ead[label=e1]{anis.matoussi@univ-lemans.fr}\thanksref{t2}}
\thankstext{t2}{The research of the first author was partially supported by the Chair {\it Financial Risks} of the {\it Risk Foundation} sponsored by Soci\'et\'e G\'en\'erale, the Chair {\it Derivatives of the Future} sponsored by the {F\'ed\'eration Bancaire Fran\c{c}aise}, and the Chair {\it Finance and Sustainable Development} sponsored by EDF and Calyon }
\address{University of Maine \\
Risk and Insurance Institute of Le Mans \\
Laboratoire  Manceau de Math\'ematiques\\ Avenue Olivier Messiaen\\
 \printead{e1}}
  \author{\fnms{Wissal} \snm{Sabbagh}\corref{}\ead[label=e2]{wissal.sabbagh@univ-lemans.fr}}

\address{ University of Maine \\
Risk and Insurance Institute of Le Mans \\
Laboratoire Manceau de Math\'ematiques\\ Avenue Olivier Messiaen \\ \printead{e2}}
 \affiliation{University of Le Mans}

\affiliation{University of Le Mans}

\runauthor{Matoussi and Sabbagh }

\begin{abstract} 
In this article, we are interested in solving numerically backward doubly stochastic differential equations (BDSDEs) with random terminal time $\tau$. The main motivations are giving a probabilistic representation of the Sobolev's solution of Dirichlet problem for semilinear SPDEs and providing the numerical scheme for such SPDEs. Thus, we study the strong approximation of this class of BDSDEs when $\tau$ is the first exit time of a forward SDE from a cylindrical domain. Euler schemes and bounds for the discrete-time approximation error are provided.

\end{abstract}


\end{frontmatter}
\section{Introduction}
\label{introduction}
Backward stochastic differential equations (BSDEs in short) are natural tools to give a probabilistic interpretation for the solution of a class of semilinear PDEs (see \cite{P91}, \cite{DP97}). By introducing in standard BSDEs a second nonlinear term driven by an external noise, we obtain Backward Doubly SDEs (BDSDEs) \cite{pp1994}, namely, 
\begin{equation}
\label{BDSDE}
Y_{t} \; = \;  
  \xi  \, + \, 
 \int_{t}^{ T} f(s,Y_{s},Z_{s})
\, ds \; + \;
 \int_{t}^{T} g(s,Y_{s},Z_{s})
\, d\W_s \; -
 \;  \int_{t}^{T} Z_{s} \,  dB_{s}, \; 0 \leq t \leq T\,  .
\end{equation}
where $(W_t)_{t\geq 0}$ and $(B_t)_{t\geq 0}$  are two  finite-dimensional independent Brownian motions. We note that the integral with respect to $B$ is a "backward It\^o integral".  In the Markovian setting, these equations can be seen as Feynman-Kac's representation of Stochastic PDEs and form a powerful tool for numerical schemes \cite{matouetal13, BM14}. These SPDEs  appear in various applications as, for instance, Zakai equations in filtering, pathwise stochastic control theory and stochastic control with partial observations. \\
Several generalizations to investigate more general nonlinear SPDEs have been developed following different approaches of the notion of weak solutions: the technique of stochastic flow (Bally and Matoussi \cite{BM01}, Matoussi et al.  \cite{MS02}); the approach based on Dirichlet forms and their associated Markov processes (Denis and Stoica \cite{DS04}, Bally, Pardoux and Stoica \cite{BPS05}, Denis, Matoussi and Stoica \cite{DMS09, DMS10}); stochastic viscosity solution for SPDEs (Buckdahn and Ma \cite{buck:ma:10a,buck:ma:10b}, Lions and Souganidis \cite{lion:soug:00, lion:soug:01, lion:soug:98}).
Above approaches have allowed the study of numerical schemes for the Sobolev solution of semilinear SPDEs via Monte-Carlo methods (time discretization and regression schemes \cite{matouetal13, BGM14a, BM14}). \\
In the case when we consider the whole space $ \Oc = \mathbb R^d$, the numerical approximation of the BSDE has already been studied in the literature by Bally \cite{B98}, Zhang \cite{Z04}, Bouchard and Touzi \cite{BT04}, Gobet, Lemor and Warin\cite{GLW05}. Bouchard and Touzi \cite{BT04} and  Zhang \cite{Z04} proposed a discrete-time numerical approximation, by step processes, for a class of decoupled FBSDEs with possible path-dependent terminal values. Zhang  \cite{Z04}  proved a regularity result on Z, which allows the use of a regular deterministic time mesh. In Bouchard and Touzi \cite{BT04}, the conditional expectations involved in their discretization scheme were computed by using the kernel regression estimation. Therefore, they used the Malliavin approach and the Monte carlo method for its computation. Crisan, Manolarakis and Touzi \cite{CMT10} proposed an improvement on the Malliavin weights. Gobet, Lemor and Warin in \cite{GLW05} proposed an explicit numerical scheme based on Monte Carlo regression on a finite basis of functions. Their approach is more efficient, because it requires only one set of paths to approximate all regression operators. 
These Monte Carlo type numerical schemes  are investigated to solve numerically the solution of semilinear PDEs. These latter methods are tractable especially when the dimension of the state process is very large unlike the finite difference method. 
For BDSDEs where the coefficient $g$   does not depend on the control variable $z$, Aman \cite{A13} proposed a numerical scheme following
the idea used by Bouchard and Touzi \cite{BT04} and obtained a convergence of order $h$ of the square of the $L^2$- error ($h$ is the discretization step in time).   
Aboura \cite{A09} studied the same numerical scheme under the same kind of hypothesis, but following Gobet et al. \cite{GLW06}. He obtained a convergence of order $h$  in time and he attempted for a Monte Carlo solver.  Bachouch et al \cite{matouetal13} have studied the rate of convergence  of the time discretization error for BDSDEs in the case when 
the coefficient $g$ depending on $ (y,z)$. They presented an implementation  and numerical tests for such Euler scheme.  Bachouch, Gobet and Matoussi \cite{BGM14a} have recently  analyzed  the regression error arising from an algorithm approximating the solution of a discrete- time BDSDEs. They have studied the rate of converge of such error in the case when the coefficients of the BDSDEs depend only on the variable $y$. \\
For BSDEs with finite random time horizon, namely, the first exit time of a forward SDEs from a domain $ \Oc$, Bouchard and Menozzi \cite{BM09} studied the Euler scheme  of these equations and provided the upper bounds for the discrete time approximations error  which is at most of order $ h^{1/2 - \varepsilon}$ where $ \varepsilon$ is any positive parameter. This  rate of convergence is due to the approximation  error of the exit time.  These results are obtained  when the domain $ \Oc$ is piecewise smooth and under a non-charateristic  boundary condition (without uniform ellipticity condition). Bouchard, Gobet and Geiss \cite{BGG13} have improved this error  which is  now at most of order $ h^{1/2}$ even if the time horizon is unbounded.\\
In this paper, we are concerned with numerical scheme for  backward doubly SDEs with random terminal time. These latter equations give the probabilistic  interpretation  for the weak-Sobolev's solutions of a class of semilinear stochastic partial differential equations (SPDEs in short) with Dirichlet null condition on the boundary of some smooth domain $\Oc \subset \mathbb R^d$ . An alternative method to solve numerically nonlinear SPDEs is an analytic one, based on time- space discretization of the SPDEs. The discretization in space can be achieved either by finite differences, or finite elements \cite{W05} and spectral Galerkin methods \cite{JK10}. But most numerical works on SPDEs have concentrated on the Euler finite-difference scheme. Very interesting results have been obtained by Gyongy and Krylov \cite{GK10}. The authors consider a symmetric finite difference scheme for a class of linear SPDE driven by an infinite dimensional Brownian motion.\\
Our contributions in this paper are as following: first of all, BDSDEs with random terminal time are introduced and results of existence and uniqueness of  such BDSDEs are established by means of some transformation to classical BSDEs  studied by  Peng \cite{P91}, Darling and Pardoux \cite{DP97} and Briand et al \cite{BDHPS03}. Next,  Euler  numerical scheme for a Forward-BDSDEs is developed where we provide  upper bounds for the discrete time approximations error  which is at most of order $ h^{1/2}$.
Then probabilistic representation for the weak solution of semilinear SPDEs with Dirichlet null condition on the boundary of the domain $ \Oc$ is given by means of solution of BDSDEs  with random terminal  time. This is done by using localization procedure and stochastic flow technics (see e.g. \cite{BM01}, \cite{MS02}, \cite{K94a, K94b}  for these flow technics).  \\
This paper is organized as follows: in Section 2, first the basic assumptions and the definitions of the solutions for BDSDEs with random terminal time are presented.  Then,  existence and  uniqueness results of  such equations are given by using fixed point theorem.  In Section 3, we develop  a discrete-time approximation of a Forward-Backward Doubly SDE with finite stopping time horizon, namely the first exit time of a forward SDE from a  domain $\Oc$. The main result of this section is providing a rate of convergence of order $ h^{1/2}$  for the square of Euler time discretization error  for   Forward-Backward Doubly SDE scheme \eqref{euldiscrete}-\eqref{Yn}. Moreover,  we relate the BDSDE in the Markovian setting  to Sobolev semilinear SPDEs with Dirichlet null condition by proving Feynman-Kac's formula in Section \ref{SPDE:section}. 
Finally, the last Section is devoted to numerical implementations and tests. 
\section{Backward doubly stochastic differential equations with random terminal time}
\label{BDSDE:section}
Any element $x\in\R^d$, $d\geq 1$, will be identified with a line vector with ith component $x^i$ and its Euclidean norm defined by $|x|=(\sum_i |x_i|^2)^{1/2}$. For each real matrix $A$, we denote by $\|A\|$ its Frobenius norm defined by $\|A\|=(\sum_{i,j}a_{i,j}^2)^{1/2}$.\\
Let $ (\Omega, {\Fc},\P)$ be a probability space, and let 
$ \{W_{t}, 0 \leq t \leq T \}$ and $ \{B_{t}, 0 \leq t \leq T \}$ 
be two mutually independent standard Brownian motions with values
in $ \R^l $ and $ \R^d $.
For each  $ 0 \leq s \leq T $, we define  
$$
{\Fc}_{s} \; := \;  {\Fc}_{s}^{B} \lor {\Fc}_{s,T}^{W},
$$
with
${\Fc}_{s}^{B}\, := \, \sigma (B_{r}; \, 0 \leq r \leq s \, ) $ and
$ {\Fc}_{s,t}^{W} \,:= \, \sigma ( W_{r} \, - \, W_{s};
s \leq r \leq t ) \lor {\Nc}$ where
$ {\Nc}$ is the class
of $\P$-null sets of $ {\Fc}$.
Note that $ \left ( {\Fc}_{t} \right )_{t \leq T} $ is not an increasing
family of $\sigma$-fields, so it is not a filtration.\\
Hereafter, let us define the spaces and the norms which will be needed for the formulation of the BDSDE with random terminal time.\\
\begin{description}
\item[-] $\mathbf{L}^p({\mathcal F}_{\tau}^{B})$ the space of $\R^k$ valued
${\mathcal F}_{\tau}^{B}$-measurable random variables $\xi$ such that
$$\begin{array}{ll}
          \|\xi\|_{L^p}^p:=\E(e^{\lambda\tau}|\xi|^p)<+\infty\,\,;
  \end{array}$$

\item[-] ${\mathcal H}^2_{k\times d}([0,T])$ the space of $\R^{k\times
d}$-valued $\Fc_t$-measurable process $Z=(Z_t)_{t\leq T}$
such that
$$\begin{array}{ll}
          \|Z\|_{{\mathcal H}^2}^2:=\E[\Int_{0}^{\tau} e^{\lambda t}|Z_t|^2dt]<+\infty\,\,;
  \end{array}$$

\item[-] ${\mathcal S}^2_{k}([0,T])$ the space of $\R^k$ valued $\Fc_t$-adapted processes $Y=(Y_t)_{t\leq T}$, with continuous paths such that
$$\|Y\|_{{\mathcal S}^2}^2:= \E[\,\underset{t\leq \tau}{\Sup}\,e^{\lambda t}|Y_t|^2]<+\infty\,\,;$$
\end{description}
We need the following assumptions:\\[0.3cm]
{\bf{Assumption (HT)}}\label{assumptionHT}
The final random time $\tau$ is an $\Fc^{B}_{\tau}$-stopping time and the final condition $ \xi$ belongs to $\mathbf{L}^2({\mathcal F}_{\tau}^{B})$.\\[0.3cm]
{\bf{Assumption (HL)}}\label{assumptionHL}
The two coefficients $f$ : 
$   \Omega\times \lbrack 0,T \rbrack \times \R^k \times 
\R^{k \times d} \to   \R^k $ and  $
g \, :  \Omega\times \lbrack 0,T \rbrack \times \R^k \times 
\R^{k \times d} \to   \R^{k \times l}$ satisfy: for all $t\in[0,T]$, $(y,z), (y^{\prime},z^{\prime})\in \R^k \times \R^{k \times d}$ and for some real 
numbers $ \alpha, \, \mu, \, \lambda,\, K>0,\, C>0,\, \displaystyle  \lambda > \frac{2K}{1 - \alpha} - 2 \mu + C$ and $ 0< \alpha < 1 $,
\begin{itemize}
\item[(i)] $f(.,y,z)$ and $g(.,y,z)$ are ${\Fc}_t $ measurable,
\item[(ii)] $ |f(t,y,z) \, - \, f(t,y^{\prime},z^{\prime}) | 
 \leq   \; K\big(|y \, - \, y^{\prime} | +\|z \, - \, z^{\prime} \|\big) $,
\item[(iii)] $ \left\langle y-y^{\prime}  \, , \, f(t,y,z) \, - \,
f(t,y^{\prime},z) \; \right\rangle \; \leq \, -\mu \, | y -y^{\prime}|^2, $
\item[(iv)] $ \|  g(t,y,z) \, - \, g(t,y^{\prime},z^{\prime}) \|^2
\; \leq   \; C \, |y \, - \, y^{\prime}|^2 \, + \, \alpha \, 
\| z \, - \, z^{\prime} \|^2  ,$
\item[(v)] $ \displaystyle \E \, \int_{0}^{\tau} e^{\lambda \, s } |f(t,0,0)|^2 \; ds \, < \, 
\infty $ and
 $ \E \, \Int_{0}^{\tau} e^{\lambda \, s } \|g(t,0,0)\|^2 \; ds \, < \, 
\infty .$ 
\end{itemize} 
Now we introduce the definition of BDSDEs with random terminal time $\tau$ and associated with $(\xi, f,  g )$.
\begin{Definition}
A solution of BDSDE $(\tau, \xi, f, g)$ is a couple 
$ \{ ( Y_{s},  Z_{s} ); \, 0 \leq s \leq T  \}\in{\mathcal S}^2_{k}([0,T])\times {\mathcal H}^2_{k\times d}([0,T]) $  such that 
$Y_t=\xi$ on the set $\{t\geq \tau\}$, $Z_t=0$  on the set $\{t>\tau\}$ and
 \begin{equation}
\label{defBDSDE r.t}
Y_{t} \; = \;  
  \xi  \, + \, 
 \Int_{t}^{\tau\wedge T} f(s,Y_{s},Z_{s})
\, ds \; + \;
 \Int_{t}^{\tau\wedge T} g(s,Y_{s},Z_{s})
\, d\W_s \; -
 \;  \Int_{t}^{\tau\wedge T} Z_{s} \,  dB_{s}, \; 0 \leq t \leq \tau\,  .
\end{equation}
\end{Definition}
We note that the integral with respect to $W$ is a "backward It\^o integral" (see Kunita \cite{K84} for the definition) and the integral with respect to $B$
is a standard forward It\^o integral. We establish  in the following theorem the existence and uniqueness of the solution for BDSDE \eqref{defBDSDE r.t} which is an extension of Peng's results \cite{P91} in the standard BSDE case. This result is also given in \cite{MPP15} (Theorem 1) for deterministic terminal time $T$ and under weaker assumptions on the coefficient $f$, namely $f$ satisfies monotonicity condition and polynomial growth in $y$. For ease of reference and completness, we give the proof of this result.
\begin{Theorem} Under the Assumptions $({\bf{HT}})$ and $({\bf{HL}})$, there exists a unique solution  $ \{ ( \, Y_{s}, \, Z_{s} \, ); \, 0 \leq s \leq T \, \}\in{\mathcal S}^2_{k}([0,T])\times {\mathcal H}^2_{k\times d}([0,T])  $ of the BDSDE \eqref{defBDSDE r.t}.
\end{Theorem}
\begin{proof}\\
{\bf{a) Uniqueness}}: Let $(Y^1,Z^1)$ and $(Y^2,Z^2)$ be two solutions of  the BDSDE  \eqref{defBDSDE r.t} and  denote by  $(\bar{Y},\bar{Z}):=(Y^1-Y^2,Z^1-Z^2)$. Applying generalized   It\^o formula (see Lemma 1.3 in \cite{pp1994})  to $e^{\lambda s}|\bar{Y}_s|^2$ yields
\be \label{unicite}
e^{\lambda t}|\bar{Y}_{t}|^2&+&\Int_{t}^{\tau\wedge T} e^{\lambda s}\big(\lambda |\bar{Y}_s|^2+ \|\bar{Z}_s\|^2\big)ds  = 2\Int_{t}^{\tau\wedge T}e^{\lambda s}\left\langle\bar{Y}_s,f(s,Y_s^1,Z_s^1)-f(s,Y_s^2,Z_s^2)\right\rangle ds\nonumber\\
& + & 2\Int_{t}^{\tau\wedge T} e^{\lambda s}\left\langle\bar{Y}_s,g(s,Y_s^1,Z_s^1)-g(s,Y_s^2,Z_s^2)\right\rangle d\W_s -2\Int_{t}^{\tau\wedge T} e^{\lambda s}\left\langle\bar{Y}_s,\bar{Z}_s\right\rangle dB_s\nonumber\\
&+&  \Int_{t}^{\tau\wedge T} e^{\lambda s} \| g(s,Y_s^1,Z_s^1)-g(s,Y_s^2,Z_s^2)\|^2ds. 
\ee
Then, taking expectation we obtain
\b* 
\E[e^{\lambda t}|\bar{Y}_{t}|^2]+\E[\Int_{t}^{\tau\wedge T}e^{\lambda s}\big(\lambda |\bar{Y}_s|^2+ \|\bar{Z}_s\|^2\big)ds ] &=& 2\E[\Int_{t}^{\tau\wedge T} e^{\lambda s}\left\langle\bar{Y}_s,f(s,Y_s^1,Z_s^1)-f(s,Y_s^2,Z_s^2)\right\rangle ds]\nonumber\\
&+& \E[\Int_{t}^{\tau\wedge T} e^{\lambda s} \| g(s,Y_s^1,Z_s^1)-g(s,Y_s^2,Z_s^2)\|^2ds]. 
\e*
From Assumption $({\bf{HL}})$ there exists $0<\varepsilon<1$ such that
$$2 \left\langle\bar{Y}_s,f(s,Y_s^1,Z_s^1)-f(s,Y_s^2,Z_s^2)\right\rangle\leq (-2\mu+\Frac{K}{1-\varepsilon})|\bar{Y}_s|^2+ (1-\varepsilon)\|\bar{ Z}_s\|^2,$$
which together with the Lipschitz continuous assumption on $g$ provide 
\b* 
\E\big[e^{\lambda t}|\bar{Y}_{t}|^2\big]+\E\big[\Int_{t}^{\tau\wedge T} e^{\lambda s}\big(\lambda |\bar{Y}_s|^2+ \|\bar{Z}_s\|^2\big)ds\big ]&\leq & \E\big[\Int_{t}^{\tau\wedge T} e^{\lambda s}(-2\mu+\Frac{K}{1-\varepsilon}+C)|\bar{Y}_s|^2ds\big]\\
&+& \E\big[\Int_{t}^{\tau\wedge T} e^{\lambda s}(\alpha+1-\varepsilon )\|\bar{Z}_s\|^2ds\big],
\e*
where $0<\alpha <1$. Consequently
\b* 
\E\big[e^{\lambda t}|\bar{Y}_{t}|^2\big]+\E\big[\Int_{t}^{\tau\wedge T} e^{\lambda s}\big((\lambda+2\mu-\Frac{K}{1-\varepsilon}-C) |\bar{Y}_s|^2+ (\varepsilon-\alpha )\|\bar{Z}_s\|^2\big)ds\big ]&\leq & 0.
\e*
Next, choosing $\varepsilon = \Frac{1 + \alpha}{2}$ and since $\lambda+2\mu-\Frac{2K}{1-\alpha}-C>0$, we conclude that 
$$Y^1_t=Y_t^2~ ~\text{and} ~~Z_t^1=Z_t^2\, , \, \P-a.s.\,, \; \forall t \in [0, \,  T].$$
{\bf{b) Existence}}: The existence of a solution will be proven in two steps. In the first step, we suppose that $g$ does not depend on $y,z$, then we are able to transform our BDSDE with data $(\tau,\xi,f,g)$ into a BSDE $(\tau,\bar{\xi},\bar{f})$, where $\bar{\xi}$ and $\bar{f}$ are explicited below. Thus, the existence is proved by appealing to the existence result for BSDEs with random terminal time estblished by Peng 1991. In the second step, we study the case when $g$ depends on $y,z$ using Picard iteration.
\\[0.1cm] 
\textit{Step 1} : Suppose that $g:=g^0$ does not depend on $y,z$, and the BDSDE \eqref{defBDSDE r.t} becomes
\be
Y_t=\xi+\Int_{t}^{\tau\wedge T} f(s,Y_s,Z_s)ds+\Int_{t}^{\tau\wedge T} g(s) d\W_s-\Int_{t}^{\tau\wedge T}  Z_sdB_s , \quad 0\leq t\leq T.
\label{BDSDElinearrandomtime}
 \ee
Denoting $$\quad \bar{Y}_t:= Y_t+\Int_0^t g(s)d\W_s,\quad \bar{\xi}:=\xi+\Int_0^\tau g(s)d\W_s,$$
we have the following BSDE
\be\label{BSDErandomtime}
\bar{Y}_{t}=\bar{\xi}+\Int_{t}^{\tau\wedge T} \bar{f}(s,\bar{Y}_s,Z_s)ds
-\Int_{t}^{\tau\wedge T} Z_{s}dB_{s}, \quad 0\leq t\leq T .
\ee
where $\bar{f}(s,y,z):= f(s,y-\Int_0^t g(s)d\W_s,z)$. We can easily check that $\bar{\xi}$ and $\bar{f}$ satisfy the same assumptions that Peng \cite{P91} (Theorem 2.2) has proved for the existence and uniqueness of the solution for the standard BSDE \eqref{BSDErandomtime}. Thus, we get the existence of the solution for the BDSDEs \eqref{BDSDElinearrandomtime}. \\[0.3cm]
\textit{Step 2} : The nonlinear case when $g$ depends on $y,z$. The solution is obtained by using  the fixed point Banach theorem. For any given $(\bar{Y},\bar{Z})\in{\mathcal H}^2_k([0,T])\times{\mathcal H}^2_{k\times d}([0,T])$, let consider the BDSDE with random terminal time:
\be
Y_t=\xi+\Int_{t}^{\tau\wedge T} f(s,Y_s,Z_s)ds+\Int_{t}^{\tau\wedge T} g(s,\bar{Y}_s,\bar{Z}_s) d\W_s-\Int_{t}^{\tau\wedge T} Z_sdB_s, \quad 0\leq t\leq T.
\label{BDSDE gen}
 \ee 
 It follows from Step 1 that the BDSDE (\ref{BDSDE gen}) has a unique solution $(Y,Z) \in {\mathcal H}^2_k([0,T])\times{\mathcal H}^2_{k\times d}([0,T])  $. Therefore, the mapping:
 \b*
 \Psi : {\mathcal H}^2_k([0,T])\times{\mathcal H}^2_{k\times d}([0,T])&\longrightarrow& {\mathcal H}^2_k([0,T])\times{\mathcal H}^2_{k\times d}([0,T])\\
   (\bar{Y},\bar{Z})&\longmapsto &\Psi(\bar{Y},\bar{Z})=(Y,Z)
 \e*
 is well defined.\\
 Next, let $(Y,Z), (Y^{'},Z^{'}), (\bar{Y},\bar{Z})$ and $(\bar{Y^{'}},\bar{Z^{'}})\in{\mathcal H}^2_k([0,T])\times{\mathcal H}^2_{k\times d}([0,T])$ such that $(Y,Z)=\Psi(\bar{Y},\bar{Z})$ and $(Y^{'},Z^{'})=\Psi(\bar{Y^{'}},\bar{Z^{'}})$ and set $\Delta\eta=\eta-\eta^{'}$ for $\eta=Y, \bar{Y}, Z, \bar{Z}, K$. Applying It\^o formula and taking expectation yield to
 \begin{align*}
 \E[e^{\lambda t}|\Delta Y_t|^2]+ \E\big[\Int_{t}^{\tau\wedge T} e^{\lambda s}\big(\lambda |\delta Y_s|^2+ \|\delta Z_s\|^2\big)ds\big ]&= 2\E[\Int_{t}^{\tau\wedge T} e^{\lambda s}\langle\Delta Y_s,f(s,Y_s,Z_s)-f(s,Y_s^{'},Z_s^{'})\rangle ds]\\
 &+ \E[\Int_{t}^{\tau\wedge T}e^{\lambda s}\|g(s,\bar{Y}_s,\bar{Z}_s)-g(s,\bar{Y^{'}}_s,\bar{Z^{'}}_s)\|^2ds].
\end{align*}
From Assumption $({\bf{HL}})$ there exists $\alpha<\varepsilon<1$ such that
$$\langle\Delta Y_s,f(s,Y_s,Z_s)-f(s,Y_s^{'},Z_s^{'})\rangle\leq (-2\mu+\Frac{K}{1-\varepsilon})|\Delta Y_s|^2+ (1-\varepsilon)\|\Delta Z_s\|^2,$$
which together with the Lipschitz continuous assumption on $g$ provide
\begin{align*}
 \E[e^{\lambda t}|\Delta Y_t|^2]&+ (\lambda+2\mu-\Frac{K}{1-\varepsilon}) \E[\Int_{t}^{\tau\wedge T} e^{\lambda s}|\delta Y_s|^2ds]+ \varepsilon\E[\Int_{t}^{\tau\wedge T} e^{\mu s}\|\Delta Z_s\|^2ds]\\
 &\leq C \E[\Int_{t}^{\tau\wedge T}e^{\mu s}|\Delta \bar{Y_s}|^2ds]+\alpha\E[\Int_{t}^{\tau\wedge T}e^{\mu s}\|\Delta \bar{Z_s}\|^2ds].
 \end{align*}
Next, choosing $\varepsilon$ such that $\lambda+2\mu-\Frac{K}{1-\varepsilon}= \Frac{\varepsilon C}{\alpha}$, we obtain
\begin{align*}
\varepsilon\big[&\Frac{C}{\alpha}\E[\Int_{t}^{\tau\wedge T} e^{\lambda s}|\Delta Y_s|^2ds] + \E[\Int_{t}^{\tau\wedge T} e^{\lambda s}\|\Delta Z_s\|^2ds]\big]\\
&\leq \alpha\Big[\Frac{C}{\alpha}\E[\Int_{t}^{\tau\wedge T} e^{\lambda s}|\Delta \bar{Y}_s|^2ds] + \E[\Int_{t}^{\tau\wedge T} e^{\lambda s}\|\Delta \bar{Z}_s\|^2ds]\big].
\end{align*}
Since $\Frac{\alpha}{\varepsilon} <1$, then $\Psi$ is a strict contraction on ${\mathcal H}^2_k([0,T]\times{\mathcal H}^2_{k\times d}([0,T])$ equipped with the norm
$$\|(Y,Z)\|^2= \Frac{C}{\alpha}\E[\Int_{0}^{\tau\wedge T} e^{\lambda s}|\Delta Y_s|^2ds] + \E[\Int_{0}^{\tau\wedge T} e^{\lambda s}\|\Delta Z_s\|^2ds].$$
Thus from Banach fixed point theorem there exists a unique pair $(Y,Z) \in  {\mathcal  H}^2_k([0,T])  \times {\mathcal H}^2_{k \times d}([0,T])$ solution of BDSDE associated to $ (\tau,\xi, f,g)$. Moreover, thanks to Assumption $({\bf{HL}})$ and standard calculations and estimates we show that $Y$ belongs to ${\mathcal S}^2_{k}([0,T]) $. \ep
\end{proof}

\section{Numerical scheme for Forward-Backward Doubly SDEs}
\label{Numeric scheme:section}
In this section, we are interested in developing a discrete-time approximation of a Forward-Backward Doubly SDE with finite stopping time horizon, namely the first exit time of a forward SDE from a cylindrical domain $D=[0,T)\times \Oc$. As usual, since we will state in the Markovian framework, we need a slight modification of the filtration.
So, we fix $t\in[0,T]$ and for each $s\in[t,T]$, we define
$$\Fc_s^t:=\Fc_{t,s}^B\vee\Fc_{s,T}^W\vee \Nc \quad \textrm{and}\quad \Gc_s^t:= \Fc_{t,s}^B\vee\Fc_{t,T}^{W}\vee \Nc,$$
where $\Fc_{t,s}^B=\sigma\{B_r-B_t, t\leq r\leq s\}$, $\Fc_{s,T}^W=\sigma\{W_r-W_s, s\leq r\leq T\}$ and $\Nc$ the class of $\P$ null sets of $\Fc$.
Note that the collection $\{\Fc_s^t, s\in[t,T]\}$ is neither increasing nor decreasing and it does not constitute a filtration. However,  $\{\Gc_s^t, s\in[t,T]\}$
is a filtration. We will omit the dependance of the filtration with respect to  the time $t$ if $t=0$.
\subsection{Formulation }
\label{FBDSDE:subsection}
For all $(t,x) \in [0,T]\times\mathbb{R}^{d}$, let $(X_{s}^{t,x})_{0\leq s\leq t}$ be the unique strong solution of the following  stochastic differential equation:
\begin{eqnarray}\label{forward}
dX_{s}^{t,x}=b(X_{s}^{t,x})ds+\sigma(X_{s}^{t,x})dB_{s},\quad s\in [t,T],\qquad X_{s}^{t,x}=x,\quad 0\leq s\leq t,
\end{eqnarray}
where $b$ and $\sigma$ are two functions on $\mathbb{R}^{d}$ with values respectively in $\mathbb{R}^{d}$ and $\mathbb{R}^{d\times d}$. We will omit the dependance of the forward process $X$ in the initial condition if it starts at time $t=0$.\\
Let $\tau^{t,x}$ be the first exit time of $(s,X_{s}^{t,x})$ from a cylindrical domain $D=[0,T)\times \Oc$ for some open bounded set $\Oc\subset\R^d$.\\
We now consider the following Markovian BDSDE with terminal random time $\tau$ associated to the data $(\Phi,f,g)$: For all $ t\leq s\leq T$,
\begin{equation}\label{BDSDE}
\left\{
\begin{array}{ll}
-dY_{s}^{t,x}&=\1_{\{s<\tau\}}f(s,X_{s}^{t,x},Y_{s}^{t,x},Z_{s}^{t,x})ds +\1_{\{s<\tau\}}g(s,X_{s}^{t,x},Y_{s}^{t,x},Z_{s}^{t,x})d\W_{s}
-Z_{s}^{t,x} dB_s,\\
\quad Y_{s}^{t,x}&=\Phi(\tau,X_{\tau}^{t,x}), \quad \tau\leq s\leq T,
\end{array}
\right.
\end{equation}
where $f$ and $\Phi$ are now two functions respectively on $[0,T]\times\mathbb{R}^{d}\times\mathbb{R}^k\times\mathbb{R}^{k\times d}$ and $\mathbb{R}^{d}$ with values in $\mathbb{R}^k$ and
 $g$ is a function on $[0,T]\times \mathbb{R}^{d}\times\mathbb{R}^k\times\mathbb{R}^ {k\times d}$ with values in $\mathbb{R}^{k\times l}$.\\[0.3cm]
 Now, we specify some conditions on the domain and the diffusion process:\\[0.3cm]
{\bf{Assumption (D)}}
$\Oc$ is an  open bounded set of $\R^d$ with a $C^2$-boundary.\\[0.3cm]
{\bf{Assumption (MHD)}} 
\begin{itemize}
\item[(i)] The matrix $a:=\sigma\sigma^{*}$ is elliptic, i.e.
there exists $\Lambda>0$ such that for all $x,\zeta\in\bar{\Oc}$ ,
\be\label{ellipticity}
\Lambda\|\zeta\|^2 \leq \zeta a(x)\zeta^{*}.
\ee
\item[(ii)]There exists a positive constant $L$ such that
\begin{eqnarray*}
 &&|b(x)-b(x')| + \|\sigma(x) - \sigma(x')\| \leq L|x-x'|,\,
\forall x, x' \in \mathbb{R}^{d}.
\end{eqnarray*}
\end{itemize}

\begin{Remark}
We mention that this smoothness assumption $({\bf{D}})$ on the domain could be weakened by considering the domain $\Oc$ as a finite intesection of smooth domains with compact boundaries and further conditions on the set of corners (see conditions $(\bf{D1})$ and $(\bf{D2})$ in \cite{BM09}). Under this weakened hypotheses, one may just assume the the matrix $a$ satisfies a non-characteristic boundary condition outside the set of corners $\Cc$ and a uniform ellipticity condition on a neighborhood of $\Cc$.
\end{Remark}
Besides, we assume that the terminal condition $\Phi$ is sufficiently smooth:\\[0.3cm]
{\bf{Assumption (MHT)}}
\b*
\Phi\in C^{1,2}([0,T]\times\R^d)\quad\text{and}\quad \|\partial_t\Phi\|+\|D\Phi\|+\|D^2\Phi\|\leq L\:\: \text{on}\: \:[0,T]\times\R^d.
\e*

We next state a strengthening of Assumption {\bf{(HL)}} in the present Markov framework:\\[0.2cm]
{\bf{Assumption ({MHL})}} There exist constants $ \alpha, \, \mu, \, \lambda,\, K>0,\, C>0,\,C^{\prime}>0,\, \displaystyle  \lambda > \frac{2K}{1 - \alpha} - 2 \mu + C$ and $ 0< \alpha < 1 $ such that
for any $(t_1,x_{1},y_{1},z_{1}),$ $(t_2,x_{2},y_{2},z_{2})\in [0,T]\times\mathbb{R}^{d}
\times\mathbb{R}^k\times \mathbb{R}^{k\times d} ,$
\begin{itemize}
\item[\rm{(i)}]
$|f(t_1,x_{1},y_{1},z_{1})-f(t_2,x_{2},y_{2},z_{2})|
\leq K \big(\sqrt{|t_{1}-t_{2}|}+|x_{1}-x_{2}|+|y_{1}-y_{2}|+\|z_{1}-z_{2} \|\big),$
\item[\rm{(ii)}]$\|g(t_1,x_{1},y_{1},z_{1})-g(t_2,x_{2},y_{2},z_{2})\|^{2}
\leq
C\big(|t_{1}-t_{2}|+|x_{1}-x_{2}|^{2}+|y_{1}-y_{2}|^{2} \big)+\alpha \|z_{1}-z_{2} \|^{2},$
\item[\rm{(iii)}] $ \left\langle y_1-y_2  \, , \, f(t_1,x_{1},y_{1},z_{1})\, - \,
f(t_1,x_{1},y_{2},z_{1}) \; \right\rangle \; \leq \, -\mu \, | y_1 -y_2|^2, $
\item[\rm{(iv)}]$\underset{0\leq t\leq T}{\Sup}(|f(t,0,0,0)|+\|g(t,0,0,0)\|)\leq C^{\prime}.$
\end{itemize}
\noindent
 \begin{Remark}
We note that the integrability condition given by Assumption {\bf{(HT)}} in Section \ref{BDSDE:section} is satisfied in this Markovian setting thanks to the smoothness of $\Phi$ (Assumption {\bf{(MHT)}}) and the fact that the exit time $\tau$, under the ellipticity condition \eqref{ellipticity} verified by the matrix $a$ (see Stroock and Varadhan \cite{SV72}), satisfies 
 $$\underset{(t,x)\in[0,T)\times\bar{\Oc}}{\Sup}\E[\exp(\lambda\tau^{t,x})]<\infty.$$
\end{Remark}
From \cite{pp1994} and \cite{K84}, the standard estimates for the solution of the Forward-Backward Doubly SDE (\ref{forward})-(\ref{BDSDE}) hold
and we remind the following theorem:
\begin{Theorem}
Under Assumptions {\bf{(MHT)}} and {\bf{(MHL)}}, there exist, for any $p\geq 2$, two positive constants $C$ and $C_p$ and an integer q such that~:
\begin{eqnarray}\label{integrability}
\E[\underset{t\leq s\leq \tau}{\Sup}|X_{s}^{t,x}|^{2}]\leq C(1+|x|^2),
\end{eqnarray}
\begin{equation}\label{apriori1}
\E\Big[\underset{t\leq s\leq \tau}{\Sup}|Y_{s}^{t,x}|^{p}+\Big(\Int_{t}^{\tau}\|Z_{s}^{t,x}\|^{2}ds \Big)^{p/2} \Big]
\leq C_p(1+|x|^{q}).
\end{equation}
\end{Theorem}
From now on, $C_L^{\eta}$ denotes a generic constant whose value may change from line to line, but which depends only on $X_0$, $L$, the constants appearing in Assumption {\bf{({MHL})}}  and some extra parameter $\eta$ (we simply write $C_L$ if it depends only on $X_0$ and $L$). Similarly, $\zeta_L^\eta$ denotes a generic non-negative random variable such that $\E[|\zeta_L^\eta|^p]\leq C_L^{\eta,p}$ for all $p\geq 1$ (we simply write $\zeta_L$ if it does not depend on the parameter $\eta$).
\subsection{Euler scheme approximation of Forward-BDSDEs}
\subsubsection{Forward Euler scheme}
In order to approximate the forward diffusion process (\ref{forward}), we use a standard Euler scheme with time step $h$, associated to a grid $$\pi:=\{t_i=ih\;;\; i\leq N\}, \qquad h:= T/N\;, \; N\in\N,$$
This approximation is defined by
\be\label{eulerforward}
X_t^N= x+\Int_0^t b(X_{\varphi(s)})ds +\Int_0^t \sigma(X_{\varphi(s)})dB_s , \qquad t\geq 0
\ee
where $\varphi(s):= \Sup\{t\in\pi : t\leq s\}$. Notice that $\varphi(t)=t_i$, for $t\in[t_i,t_{i+1})$ and the continuous approximation (\ref{eulerforward}) is equivalent to the following discrete approximation
\begin{equation}\label{euldiscrete}
\left\{
\begin{array}{ll}
X_0^N=x,\\
X_{t_{i+1}}^N= X_{t_{i}}^N+b(X_{t_{i}}^N)(t_{i+1}-t_{i})+\sigma(X_{t_{i}}^N)(B_{t_{i+1}}-B_{t_{i}}),\quad i\leq N.
\end{array}
\right.
\end{equation}
Then, we approximate the exit time $\tau$ by the first time of the Euler scheme $(t,X_t^N)_{t\in\pi}$ from $D$ on the grid $\pi$:
\b*
\bar{\tau}:=\Inf\{t\in\pi : X_t^N\notin\Oc\}\wedge T.
\e*
\begin{Remark}
One may approximate the exit time $\tau$ by its continuous version $\tilde{\tau}$ which is defined as the first exit time of the Euler scheme $(t,X_t^N)$, namely 
$$\tilde{\tau}:=\Inf\{t\in[0,T] : X_t^N\notin\Oc\}\wedge T.$$
However, this approximation requires more regularity on the boundary of $\Oc$ (see e.g. \cite{Gob98,Gob2000}).
\end{Remark}
The upper bound estimates for the error due to the approximation of $\tau$ by $\bar{\tau}$  was proved by Bouchard and Menozzi \cite{BM09} for the weak version of such estimate and Gobet \cite{Gob98,Gob2000} for the strong one. Recently, Bouchard, Geiss and Gobet \cite{BGG13} have improved the following  $L^1$-strong error:
\begin{Theorem}
Assume that {\bf{(MHD)}} and {\bf{(D)}} hold. Then, there exists $C_L>0$ such that
\be\label{proxytau:BGG}
\E [|\tau-\bar{\tau}|]\leq C_L h^{1/2}.
\ee
\end{Theorem}
\begin{Remark}
Let us mention that the upper bound estimates for the error due to the approximation of $\tau$ by $\bar{\tau}$  proved by Bouchard and Menozzi \cite{BM09} for the weak version of such estimate  is as following: for any $\varepsilon\in(0,1)$ and each positive random variable $\zeta$ satisfying $\E[(\zeta_L)^p]\leq C_L^p$ for all $p\geq 1$, there exists $C_L^\varepsilon>0$ such that
\be\label{proxytau:faible}
\E\big[\E[\zeta_L|\tau-\bar{\tau}||\Fc^B_{\tau_{+}\wedge\bar{\tau}}]^2\big]\leq C_L^\varepsilon h^{1-\varepsilon},
\ee 
where $\tau_{+}$ is the next time after $\tau$ in the grid $\pi$ such that $\tau_+:=\Inf\{t\in \pi : \tau\leq t\}$.\\
For the strong estimate error, Gobet \cite{Gob98,Gob2000}  has proved that, for each $\varepsilon\in(0,1/2)$, there exists $C_L^\varepsilon>0$ such that 
\be\label{proxytau:fort}
\E[|\tau-\bar{\tau}|]\leq C_L^\varepsilon h^{1/2-\varepsilon}.
\ee
\end{Remark}
\subsubsection{Euler scheme for BDSDEs}
Regarding the approximation of (\ref{BDSDE}), we adapt the approach of \cite{matouetal13}. 
We define recursively (in a backward manner) the discrete-time process $(Y^N,Z^N)$ on the time grid $\pi$ by 
\begin{equation}\label{conterminale}
Y_{T}^{N} = \Phi(\bar{\tau},X_{\bar{\tau}}^{N}),
\end{equation}
and for $i=N-1,\ldots,0$, we set
\begin{equation}\label{Zn}
Z_{t_{i}}^{N} =h^{-1} \E_{t_{i}}\Bigg[(Y_{t_{i+1}}^{N}
+ g(t_{i+1},\Theta_{i+1}^{N})\Delta W_{i})\Delta B^{\top}_{i}\Bigg],
\end{equation}
\begin{equation}\label{Yn}
Y_{t_{i}}^{N} = \E_{t_{i}}[Y_{t_{i+1}}^{N}] + \1_{\{t_i<\bar{\tau}\}}h \E_{t_{i}}[f(t_{i},\Theta_{i}^{N})]
+\1_{\{t_i<\bar{\tau}\}}\E_{t_{i}}[g(t_{i+1},\Theta_{i+1}^{N})\Delta W_{i}],
\end{equation}

where
\begin{eqnarray*}
\Theta_{i}^{N}:=(X_{t_{i}}^{N},Y_{t_{i}}^{N},Z_{t_{i}}^{N})\;,\; \Delta W_{i}=W_{t_{i+1}}-W_{t_{i}}\; ,\;\Delta B_{i}=B_{t_{i+1}}-B_{t_{i}}.
\end{eqnarray*}
$\top$ denotes the transposition operator and
$\E_{t_{i}}$ denotes the conditional expectations over the $\sigma$-algebra $\Fc_{t_{i}}^0$. The above conditional expectations  are well defined at each step of the algorithm.\\
Observe that $Y_{t_{i}}^{N}\1_{\{t_i\geq\bar{\tau}\}}=\Phi(\bar{\tau},X_{\bar{\tau}}^{N})
\1_{\{t_i\geq\bar{\tau}\}}$ and $Z_{t_{i}}^{N}\1_{\{t_i\geq\bar{\tau}\}}=0$. One can easily check that $$Y_{t_{i+1}}^{N}
+ g(t_{i+1},\Theta_{i+1}^{N})\Delta W_{i}\in L^2(\Fc_{t_{i+1}})$$ for all $0\leq i<N$ under the Lipschitz continuous assumption. Then an obvious extension of It\^o martingale representation theorem yields the existence of the $\Gc_t$-progressively measurable and square integrable process $Z^N$ satisfying, for all $i<N$
\b*
Y_{t_{i+1}}^{N} + g(t_{i+1},\Theta_{i+1}^{N})\Delta W_{i}= \E_{t_{i}}[Y_{t_{i+1}}^{N} + g(t_{i+1},\Theta_{i+1}^{N})\Delta W_{i}]+\Int_{t_i}^{t_{i+1}}Z^N_s dB_s .
\e*
Following the arguments of Pardoux and Peng \cite{pp1994} (see page 213), we can prove that in fact $Z^N$ is $\Fc_t$-progressively measurable thanks to the independance of the increments of $B$ and the two Browian motions $B$ and $W$.\\
This allows us to consider a continuous-time extension of $Y^N$ in $\Sc^2$ defined on $[0,T]$ by 
\be \label{Yapprox}
Y_t^N &=&  \Phi(\bar{\tau},X_{\bar{\tau}}^{N})+\Int_t^T \1_{\{s<\bar{\tau}\}} f(\varphi(s),\Theta_{\varphi(s)}^{N})ds+\Int_t^T \1_{\{s<\bar{\tau}\}} g(\psi(s),\Theta_{\psi(s)}^{N})d\W_s-\Int_t^T Z_s^N dB_s,\nonumber\\
& &
\ee
where $\psi(s):= \Inf\{t\in\pi : t\geq s\}$.
\begin{Remark}\label{remZ}
Observe that $Z_s=0$ on $]\tau,T]$ and $Z_s^N=0$ on $]\bar{\tau},T]$. For later use, note also that 
\be \label{Zapprox}
Z_{t_i}^{N}= h^{-1}\E_{t_i}[\Int_{t_i}^{t_{i+1}}Z_s^N ds ]\: , \quad i<N.
\ee
\end{Remark}

In order to prove (\ref{errtau}) of Proposition \ref{properr2}, we need the following lemma.
\begin{Lemma}\label{lemmcontrol}
Let Assumptions {\bf{(MHL)}} and {\bf{(MHT)}} hold. Then,
\be 
\underset{i<N}{\max} (|Y_{t_i}^N|+\sqrt{h}\| Z_{t_i}^N\|)\leq \zeta_L \quad\text{and}\quad \|Y^N\|_{\Sc^2}+ \|Z^N_{\varphi}\|_{\Hc^2}+ \|Z^N_{\psi}\|_{\Hc^2}\leq C_L.
\ee
\end{Lemma}

\subsubsection{Upper bounds for the discrete-time approximation error }
In this section, we provide bounds for the (square of the) discrete-time approximation error up to a stopping time $\theta\leq T$ $\P$-a.s. defined  as 
\be
\text{Err}(h)_{\theta}^2:=\underset{i<N}{\max}\,\E\big[\underset{t\in[t_i,t_{i+1}]}{\Sup}\1_{\{t<\theta\}} |Y_t-Y_t^N|^2\big] + \E\big[\Int _0^\theta\|Z_t-Z_{\varphi(t)}^N\|^2 dt\big],
\ee
where we recall $\varphi(s):= \Sup\{t\in\pi : t\leq s\}$.\\

We first recall some standard controls on $X$, $(Y,Z)$ and $X^N$.
\begin{Proposition}\label{control}
 Let Assumptions {\bf{(MHL)}}, {\bf{(MHT)}} and {\bf{(MHD)}} hold. Fix $p\geq 2$. Let $\vartheta$ be a stopping time with values in $[0,T]$. Then,
$$\E\Big[\underset{t\in[\vartheta,T]}{\Sup}|Y_t|^p +\big(\Int_\vartheta^T \|Z_t\|^2dt\big)^{p/2}\Big]\leq C_L^p(1+|X_\vartheta|^p),$$
and
$$\E\Big[\underset{t\in[\vartheta,T]}{\Sup}(|X_t|^p +|X_t^N|^p)|\Fc^B_{0,\vartheta}\Big]\leq \zeta_L^p.$$
Moreover,
$$\underset{i<N}{\max}\,\E\big[\underset{t\in[t_i,t_{i+1}]}{\Sup}(|X_t-X_{t_i}|^p+|X_t^N-X_{t_i}^N|^p)\big]+\E\big[\underset{t\in[0,T]}{\Sup}(|X_t-X_{t}^N|^p)\big]\leq C_L^ph^{p/2},$$
$$\P\big[\underset{t\in[0,T]}{\Sup}(|X_t^N-X_{\varphi(t)}^N|>r\big]\leq C_{L }r^{-4} h, \quad r>0,$$
and, if $\theta$ is a stopping time with values in $[0,T]$ such that $\vartheta\leq \theta\leq \vartheta+h$ $\P$-a.s., then
$$\E\big[|X_{\theta}^N-X_{\vartheta}^N|^p + |X_{\theta}-X_{\vartheta}|^p|\Fc^B_{0,\vartheta}\big]\leq \zeta_L^p h^{p/2}.$$
\end{Proposition}
\begin{Remark}
Let $\vartheta\leq\theta$ $\P$-a.s. be two stopping times with values in $\pi$ and $\bar{Z}_{t_i}$ be the best approximation of $(Z_t)_{t_i\leq t\leq t_{i+1}}$ by $\Fc_{{t_i}}$-measurable random variable in the following sense
\be\label{approxZ}
\bar{Z}_{t_i}:= h^{-1}\E_{t_i}[\Int_{t_i}^{t_{i+1}}Z_s ds ]\: , \quad i<N.
\ee
Then, recalling that $t_{i+1}-t_{i}=h$, it follows from (\ref{approxZ}),(\ref{Zapprox}) and Jensen's inequality that 
\begin{align}\label{Zest}
\begin{split}
\E\Big[\Int_\vartheta^\theta \|\bar{Z}_{\varphi(s)}-Z_{\varphi(s)}^N\|^2ds\Big]&=\sum_{i<N} \E\Big[\Int_{t_i}^{t_{i+1}}\1_{\{\vartheta\leq t_i\leq \theta\}} 
\big\|\E_{t_i}\big[h^{-1}\Int_{t_i}^{t_{i+1}}(Z_{u}-Z_u^N) du\big]\big\|^2ds
\Big]\\
&\leq \sum_{i<N} \E\Big[\Int_{t_i}^{t_{i+1}}\1_{\{\vartheta\leq t_i\leq \theta\}} h^{-1}\Int_{t_i}^{t_{i+1}}\|Z_u-Z_u^N\|^2 du\:ds\Big]\\
&\leq \E\Big[\Int_\vartheta^\theta \|Z_s-Z_{s}^N\|^2ds\Big]
\end{split}
\end{align}
Observe that the above inequality does not apply if $\vartheta$ and $\theta$ do not take values in $\pi$. This explains why it is easier to work with $\tau_+$,  instead of $\tau$, that is, work on $\text{Err}(h)_{\tau_+\wedge\bar{\tau}}^2$ instead of $\text{Err}(h)_{\tau\wedge\bar{\tau}}^2$.
\end{Remark}
Now we state an upper bound result for some stopping time $\theta$ with values in $\pi$.
\begin{Theorem}\label{errtheta}
Assume that Assumptions \textbf{(MHL)}, \textbf{(MHD)} and \textbf{(MHT)}   hold, and define 
$$\Rc(Y)_{\Sc^2}^{\pi}:=\underset{i<N}{\max}\, \E\big[\underset{t\in[t_i,t_{i+1}]}{\Sup}|Y_t-Y_{t_i}|^2\big]\quad ,\quad \Rc(Z)_{\Hc^2}^{\pi}:=\E\big[\Int _0^T\|Z_t-\bar{Z}_{\varphi(t)}\|^2 dt\big]$$

Then for all stopping times $\theta$ with values in $\pi$, we have
\begin{align}
\text{Err}(h)_{\theta}^2\leq C_L &\bigg(h +\E[|Y_{\theta}-Y_{\theta}^N|^2]+\Rc(Y)_{\Sc^2}^{\pi}+\Rc(Z)_{\Hc^2}^{\pi}\nonumber\\
&+\E\big[\Int _0^T\|Z_t-\bar{Z}_{\psi(t)}\|^2 dt\big]+\E\big[\Int_{\bar{\tau}\wedge\tau\wedge\theta}^{(\bar{\tau}\vee\tau)\wedge\theta}(\zeta_L+\1_{\{\bar{\tau}<\tau\}}\|Z_s\|^2)ds\big]\bigg).
\end{align}
\end{Theorem}

\begin{proof}
Equations  (\ref{BDSDE}) and (\ref{Yapprox}), the generalized  Ito's lemma (see  Lemma 1.3 in \cite{pp1994}) applied to $(Y-Y^N)^2$ on $[t\wedge\theta,t_{i+1}\wedge\theta]$ for $t\in[t_{i},t_{i+1}]$ and $i<N$, and  taking expectation  yield to  
\begin{align*}
\Delta_{t,t_{i+1}}^\theta &:= \E\big[|Y_{t\wedge\theta}-Y_{t\wedge\theta}^N|^2 +\Int_{t\wedge\theta}^{t_{i+1}\wedge\theta}\|Z_s-Z_s^N\|^2ds\big]\\
& = \E\big[|Y_{t_{i+1}\wedge\theta}-Y_{t_{i+1}\wedge\theta}^N|^2 \big]+ \E\big[2\Int_{t\wedge\theta}^{t_{i+1}\wedge\theta}(Y_{s}-Y_{s}^N)(\1_{\{s<\tau\}}f(\Theta_s)-\1_{\{s<\bar{\tau}\}}f(\Theta_{\varphi(s)}^N))ds\big]\\
&\quad +\E\big[\Int_{t\wedge\theta}^{t_{i+1}\wedge\theta}\|\1_{\{s<\tau\}}g(\Theta_s)-\1_{\{s<\bar{\tau}\}}g(\Theta_{\psi(s)}^N)\|^2ds\big] ,
\end{align*}
where $\Theta_{s}:=(X_s,Y_s,Z_s)$. Using the fact that $\1_{\{s<\tau\}}\leq \1_{\{s<\bar{\tau}\}}+\1_{\{\tau\leq s<\bar{\tau}\}}+\1_{\{\bar{\tau}\leq s<\tau\}}$ and the inequality $2ab\leq \varepsilon a^2+\varepsilon^{-1}b^2$, we then deduce that for $\varepsilon>0$ to be chosen later,
\begin{align*}
\Delta_{t,t_{i+1}}^\theta &\leq \E\big[|Y_{t_{i+1}\wedge\theta}-Y_{t_{i+1}\wedge\theta}^N|^2 \big]+\varepsilon\E\big[\Int_{t\wedge\theta}^{t_{i+1}\wedge\theta}|Y_{s}-Y_{s}^N|^2ds\big]\\
&\quad+\varepsilon^{-1}\E\big[\Int_{t\wedge\theta}^{t_{i+1}\wedge\theta}
\1_{\{s<\bar{\tau}\}}(f(\Theta_s)-f(\Theta_{\varphi(s)}^N))^2ds+\Int_{t\wedge\theta}^{t_{i+1}\wedge\theta}\1_{\{\bar{\tau}\leq s<\tau\}}(f(\Theta_s))^2ds\big]\\
&\quad+\varepsilon^{-1}\E\big[\Int_{t\wedge\theta}^{t_{i+1}\wedge\theta}\1_{\{\tau\leq s<\bar{\tau}\}}(f(\Theta_s))^2ds\big]+\E\big[\Int_{t\wedge\theta}^{t_{i+1}\wedge\theta}
\1_{\{s<\bar{\tau}\}}\|g(\Theta_s)-g(\Theta_{\psi(s)}^N)\|^2ds\big]\\
&\quad+\E\big[\Int_{t\wedge\theta}^{t_{i+1}\wedge\theta}\1_{\{\bar{\tau}\leq s<\tau\}}\|g(\Theta_s)\|^2ds\big]
+\E\big[\Int_{t\wedge\theta}^{t_{i+1}\wedge\theta}\1_{\{\tau\leq s<\bar{\tau}\}}\|g(\Theta_s)\|^2ds\big]
\end{align*}
Recall from Remark \ref{remZ} that $Z=0$ on $]\tau,T]$. Since $Y_t=\Phi(\tau,X_\tau)$ on $\{t\geq \tau\}$, we then deduce from the Lipschitz continuous assumption \textbf{(MHL)} that 
\b*
\Delta_{t,t_{i+1}}^\theta &\leq & \E\big[|Y_{t_{i+1}\wedge\theta}-Y_{t_{i+1}\wedge\theta}^N|^2 \big]+\varepsilon\E\big[\Int_{t\wedge\theta}^{t_{i+1}\wedge\theta}|Y_{s}-Y_{s}^N|^2ds\big]\nonumber\\
&\quad &+C_L\varepsilon^{-1}\E\big[\Int_{t\wedge\theta}^{t_{i+1}\wedge\theta}
\1_{\{s<\bar{\tau}\}}(|X_s-X_{\varphi(s)}^N|^2+|Y_s-Y_{\varphi(s)}^N|^2+\|Z_s-Z_{\varphi(s)}^N\|^2)ds\big]\nonumber\\
&\quad &+C_L(\varepsilon^{-1}+1)\E\big[\Int_{t\wedge\theta}^{t_{i+1}\wedge\theta}
\1_{\{\bar{\tau}\leq s<\tau\}}(|X_s|^2+|Y_s|^2)ds\big]\nonumber\\
&\quad &+C_L(\varepsilon^{-1}+1)\E\big[\Int_{t\wedge\theta}^{t_{i+1}\wedge\theta}
\1_{\{\tau\leq s<\bar{\tau}\}}(|X_\tau|^2+|\Phi(\tau,X_\tau)|^2)ds\big]\nonumber\\
&\quad &+(C_L\varepsilon^{-1}+\alpha)\E\big[\Int_{t\wedge\theta}^{t_{i+1}\wedge\theta}
\1_{\{\bar{\tau}\leq s<\tau\}}\|Z_s\|^2ds\big]\nonumber\\
&\quad &+\E\big[\Int_{t\wedge\theta}^{t_{i+1}\wedge\theta}
\1_{\{s<\bar{\tau}\}}(C_L|X_s-X_{\psi(s)}^N|^2+C_L|Y_s-Y_{\psi(s)}^N|^2+\alpha\|Z_s-Z_{\psi(s)}^N\|^2)ds\big].
\e*
Now, appealing to Proposition \ref{control} yields to
\begin{align*}
&\Delta_{t,t_{i+1}}^\theta\leq \E\big[|Y_{t_{i+1}\wedge\theta}-Y_{t_{i+1}\wedge\theta}^N|^2 \big]+\varepsilon\E\big[\Int_{t\wedge\theta}^{t_{i+1}\wedge\theta}|Y_{s}-Y_{s}^N|^2ds\big]\nonumber\\
&+C_L\varepsilon^{-1}\E\big[\Int_{t\wedge\theta}^{t_{i+1}\wedge\theta}
(h+|Y_s-Y_{\varphi(s)}|^2+|Y_{\varphi(s)}-Y_{\varphi(s)}^N|^2+\|Z_s-\bar{Z}_{\varphi(s)}\|^2+\|\bar{Z}_{\varphi(s)}-Z_{\varphi(s)}^N\|^2)ds\big]\nonumber\\
&+C_L(\varepsilon^{-1}+1)\E\big[\Int_{t\wedge\theta}^{t_{i+1}\wedge\theta}
\1_{\{\bar{\tau\wedge}\tau\leq s<\tau\vee\bar{\tau}\}} \zeta_L ds\big]+(C_L\varepsilon^{-1}+\alpha)\E\big[\Int_{t\wedge\theta}^{t_{i+1}\wedge\theta}
\1_{\{\bar{\tau}\leq s<\tau\}}\|Z_s\|^2ds\big]\nonumber\\
&+\E\big[\Int_{t\wedge\theta}^{t_{i+1}\wedge\theta}
(C_Lh+C_L|Y_s-Y_{\psi(s)}|^2+C_L|Y_{\psi(s)}-Y_{\psi(s)}^N|^2+\alpha\|Z_s-\bar{Z}_{\psi(s)}\|^2+\alpha\|\bar{Z}_{\psi(s)}-Z_{\psi(s)}^N\|^2)ds\big].
\end{align*}
Next, we obtain from the definition of $\varphi$
\begin{align}\label{delta_t}
&\Delta_{t,t_{i+1}}^\theta \leq \E\big[|Y_{t_{i+1}\wedge\theta}-Y_{t_{i+1}\wedge\theta}^N|^2 \big]+\varepsilon\E\big[\Int_{t\wedge\theta}^{t_{i+1}\wedge\theta}|Y_{s}-Y_{s}^N|^2ds\big]\nonumber\\
&\quad+C_L(\varepsilon^{-1}+1)\E\big[h|Y_{t_{i}\wedge\theta}-Y_{t_{i}\wedge\theta}^N|^2+h|Y_{t_{i+1}\wedge\theta}-Y_{t_{i+1}\wedge\theta}^N|^2+\Int_{t\wedge\theta}^{t_{i+1}\wedge\theta}(|Y_s-Y_{\varphi(s)}|^2+|Y_s-Y_{\psi(s)}|^2ds)\big]\nonumber\\
&\quad+C_L(\varepsilon^{-1}+1)\E\big[\Int_{t\wedge\theta}^{t_{i+1}\wedge\theta}hds\big]+C_L\varepsilon^{-1}\E\big[\Int_{t\wedge\theta}^{t_{i+1}\wedge\theta}(\|Z_s-\bar{Z}_{\varphi(s)}\|^2+\|\bar{Z}_{\varphi(s)}-Z_{\varphi(s)}^N\|^2)ds\big]\nonumber\\
&\quad+\alpha\E\big[\Int_{t\wedge\theta}^{t_{i+1}\wedge\theta}(\|Z_s-\bar{Z}_{\psi(s)}\|^2+\|\bar{Z}_{\psi(s)}-Z_{\psi(s)}^N\|^2)ds\big]\nonumber\\
&\quad+C_L(\varepsilon^{-1}+1)\E\big[\Int_{t\wedge\theta}^{t_{i+1}\wedge\theta}
\1_{\{\bar{\tau\wedge}\tau\leq s<\tau\vee\bar{\tau}\}} \zeta_L ds\big]+(C_L\varepsilon^{-1}+\alpha)\E\big[\Int_{t\wedge\theta}^{t_{i+1}\wedge\theta}
\1_{\{\bar{\tau}\leq s<\tau\}}\|Z_s\|^2ds\big].
\end{align}
It then follows from Gronwall's lemma that
\begin{align}\label{Gronwall}
\E\big[|Y_{t\wedge\theta}- & Y_{t\wedge\theta}^N|^2\big]\nonumber
 \leq (1+C_L(\varepsilon^{-1}+1)h+C_L^{\varepsilon}h)\E\big[|Y_{t_{i+1}\wedge\theta}-Y_{t_{i+1}\wedge\theta}^N|^2 \big]\nonumber\\
&\quad+(C_L(\varepsilon^{-1}+1)+C_L^{\varepsilon}h)\E\big[h|Y_{t_{i}\wedge\theta}-Y_{t_{i}\wedge\theta}^N|^2+\Int_{t\wedge\theta}^{t_{i+1}\wedge\theta}(|Y_s-Y_{\varphi(s)}|^2+|Y_s-Y_{\psi(s)}|^2)ds\big]\nonumber\\
&\quad+(C_L(\varepsilon^{-1}+1)+C_L^{\varepsilon}h)\E\big[\Int_{t\wedge\theta}^{t_{i+1}\wedge\theta}hds\big] \nonumber\\
& \quad +(C_L\varepsilon^{-1} +C_L^{\varepsilon}h)\E\big[\Int_{t\wedge\theta}^{t_{i+1}\wedge\theta}(\|Z_s-\bar{Z}_{\varphi(s)}\|^2+\|\bar{Z}_{\varphi(s)}-Z_{\varphi(s)}^N\|^2)ds\big]\nonumber\\
&\quad+(\alpha+C_L^{\varepsilon}h)\E\big[\Int_{t\wedge\theta}^{t_{i+1}\wedge\theta}(\|Z_s-\bar{Z}_{\psi(s)}\|^2+\|\bar{Z}_{\psi(s)}-Z_{\psi(s)}^N\|^2)ds\big]\nonumber\\
&\quad+(C_L(\varepsilon^{-1}+1)+C_L^{\varepsilon}h)\E\big[\Int_{t\wedge\theta}^{t_{i+1}\wedge\theta}
\1_{\{\bar{\tau\wedge}\tau\leq s<\tau\vee\bar{\tau}\}} \zeta_L ds\big]\nonumber\\
&  \quad +(C_L\varepsilon^{-1}+\alpha+C_L^{\varepsilon}h)\E\big[\Int_{t\wedge\theta}^{t_{i+1}\wedge\theta}
\1_{\{\bar{\tau}\leq s<\tau\}}\|Z_s\|^2ds\big].\nonumber\\
\end{align}
Then, by taking $t=t_i$ in (\ref{delta_t}), using (\ref{Gronwall}) to estimate the second term in the right-hand side of (\ref{delta_t}) and recalling Remark \ref{remZ} we have for $\varepsilon>0$ sufficiently large, depending on the constants $C_L$, and $h$ small 
\begin{align*}
& \quad \Delta_{t_{i},t_{i+1}}^\theta \leq (1+C_Lh)\E\big[|Y_{t_{i+1}\wedge\theta}-Y_{t_{i+1}\wedge\theta}^N|^2 \big]\\
&\quad+C_L\E\big[\Int_{t_{i}\wedge\theta}^{t_{i+1}\wedge\theta}(h+|Y_s-Y_{\varphi(s)}|^2+|Y_s-Y_{\psi(s)}|^2ds)\big]\\
&\quad+C_L\E\big[\Int_{t_{i}\wedge\theta}^{t_{i+1}\wedge\theta}\|Z_s-\bar{Z}_{\varphi(s)}\|^2ds\big]+C_L\E\big[\Int_{t_{i}\wedge\theta}^{t_{i+1}\wedge\theta}\|Z_s-\bar{Z}_{\psi(s)}\|^2ds\big]\nonumber\\
&\quad+C_L\E\big[\Int_{t_{i}\wedge\theta}^{t_{i+1}\wedge\theta}
\1_{\{\bar{\tau\wedge}\tau\leq s<\tau\vee\bar{\tau}\}} \zeta_L ds\big]+C_L\E\big[\Int_{t_{i}\wedge\theta}^{t_{i+1}\wedge\theta}
\1_{\{\bar{\tau}\leq s<\tau\}}\|Z_s\|^2ds\big].\nonumber
\end{align*}
Thus, from the following estimate
\begin{align*}
\E[|Y_s-Y_{\psi(s)}|^2]&\leq \E[\underset{t_i\leq s\leq t_{i+1}}{\Sup}|Y_s-Y_{\psi(s)}|^2]\\
&\leq  C_{L}(1+|x|)h,
\end{align*}
we  conclude that
\begin{align*}
\Delta^\theta &:=\underset{i<N}{\Max} \E\big[|Y_{t_{i}\wedge\theta}-Y_{t_{i}\wedge\theta}^N|^2 +\Int_{0}^{\theta}\|Z_s-Z_s^N\|^2ds\big]\\
&\leq C_L\big(\E\big[|Y_{\theta}-Y_{\theta}^N|^2 \big]+h+\Rc(Y)_{\Sc^2}^{\pi}+\Rc(Z)_{\Hc^2}^{\pi}+ E\big[\Int _0^T\|Z_t-\bar{Z}_{\psi(t)}\|^2 dt\big]\big)\\
& + C_L\E\big[\zeta_L|\bar{\tau}\wedge\theta-\tau\wedge\theta|+\Int_{0}^{\theta}\1_{\{\bar{\tau}\leq s<\tau\}}\|Z_s\|^2ds\big].
\end{align*}
We finish the proof by using again Remark \ref{remZ} to obtain
\begin{align}
\E\bigg[&\Int_0^{\theta}\|Z_s-Z_{\varphi(s)}^N\|^2ds\bigg]\nonumber\\
&\leq C_L\bigg(\E\bigg[\Int_0^{\theta}\|\bar{Z}_{\varphi(s)}-Z_{\varphi(s)}^N\|^2ds\bigg]+\E\bigg[\Int_0^{T}\|Z_s-\bar{Z}_{\varphi(s)}\|^2ds\bigg]\bigg)\\
&\leq C_L\bigg(\E\bigg[\Int_0^{\theta}\|Z_{s}-Z_s^N\|^2ds\bigg]+\Rc(Z)_{\Hc^2}^{\pi}\bigg),\nonumber
\end{align}
which implies the required result, by the definition of $\text{Err}(h)_{\theta}^2$ in (\ref{errtheta}).
\ep
\end{proof}

\begin{Proposition}\label{properr2}
Let Assumptions {\bf{(MHL)}}, {\bf{(MHD)}} and {\bf{(MHT)}} hold. There then exist $C_L>0$ and a positive random variable $\zeta_L$ satisfying $\E[(\zeta_L)^p]\leq C_L^p$ for all $p\geq 2$ such that 
\be\label{errT}
\begin{split}
\text{Err}(h)_{T}^2\leq C_L& \Big(h+\Rc(Y)_{\Sc^2}^{\pi}+\Rc(Z)_{\Hc^2}^{\pi}+\E\big[\zeta_L|\tau-\bar{\tau}|+\1_{\{\bar{\tau}<\tau\}}\Int_{\bar{\tau}}^{\tau}\|Z_s\|^2ds\big]\\
&+\E\big[\Int _0^T\|Z_t-\bar{Z}_{\psi(t)}\|^2 dt\big]\Big).
\end{split}
\ee 
and
\begin{align}\label{errtau}
\begin{split}
\text{Err}(h)_{\tau\wedge\bar{\tau}}^2 \leq \text{Err}(h)_{\tau_{+}\wedge\bar{\tau}}^2 \leq C_L & \Big(h+\Rc(Y)_{\Sc^2}^{\pi}+\Rc(Z)_{\Hc^2}^{\pi}+\E\big[\zeta_L|\tau-\bar{\tau}|+\1_{\{\bar{\tau}<\tau\}}\Int_{\bar{\tau}}^{\tau}\|Z_s\|^2ds\big]\\
&+\E\big[\Int _0^T\|Z_t-\bar{Z}_{\psi(t)}\|^2 dt\big]\Big),
\end{split}
\end{align}
where we recall $\tau_+$ is the next time after $\tau$ in the grid $\pi$ such that $\tau_+:=\Inf\{t\in \pi : \tau\leq t\}$.
\end{Proposition}
\begin{Remark}
Note that we shall control $\text{Err}(h)_{\tau\wedge\bar{\tau}}^2$ through the slightly stronger term $ \text{Err}(h)_{\tau_{+}\wedge\bar{\tau}}^2$. This will allow us to work with stopping times with values in the grid $\pi$ in order to be able to apply \eqref{Zest}, which will be technically easier.
\end{Remark}
\begin{proof}\\
(i) First to prove (\ref{errT}), it suffices to apply Theorem \ref{errtheta} for $\theta=T$ and observe that the Lipschitz continuity of $\Phi$ implies that
\begin{align*}
\E[|&\Phi(\tau,X_{\tau})-\Phi(\bar{\tau},X^N_{\bar{\tau}})|^2]\\
&\leq C_L\E\big[|\tau-\bar{\tau}|^2+|X_{\bar{\tau}}-X^N_{\bar{\tau}}|^2+\Big|\Int_{\tau\wedge\bar{\tau}}^{\tau\vee\bar{\tau}}b(X_s)ds+\Int_{\tau\wedge
\bar{\tau}}^{\tau\vee\bar{\tau}}\sigma(X_s)dB_s\Big|^2\big],
\end{align*}
where $|\tau-\bar{\tau}|^2\leq T|\tau-\bar{\tau}|$, $\E[|X_{\bar{\tau}}-X^N_{\bar{\tau}}|^2]\leq C_L h$ by Proposition \ref{control} and 
\b*
\E\Big[\Big|\Int_{\tau\wedge\bar{\tau}}^{\tau\vee\bar{\tau}}b(X_s)ds
+\Int_{\tau\wedge\bar{\tau}}^{\tau\vee\bar{\tau}}\sigma(X_s)dB_s\Big|^2\Big]\leq \E[\zeta_L|\tau-\bar{\tau}|]
\e* by Doob's inequality, \textbf{(MHD)} and Proposition \ref{control} again.\\
(ii) We now prove the upper bound (\ref{errtau}). We have by applying Theorem \ref{errtheta} to $\theta=\tau_{+}\wedge\bar{\tau}$
\b*
\text{Err}(h)_{\tau_{+}\wedge\bar{\tau}}^2\leq C_L \big(h +\E[|Y_{\tau_{+}\wedge\bar{\tau}}-Y_{\tau_{+}\wedge\bar{\tau}}^N|^2]+\Rc(Y)_{\Sc^2}^{\pi}+\Rc(Z)_{\Hc^2}^{\pi}\big).
\e*
It remains to show that 
\be
\E[|Y_{\tau_{+}\wedge\bar{\tau}}-Y_{\tau_{+}\wedge\bar{\tau}}^N|^2]\leq  C_L \big(h+\E\big[\zeta_L|\tau-\bar{\tau}|+\1_{\{\bar{\tau}<\tau\}}\Int_{\bar{\tau}}^{\tau}\|Z_s\|^2ds\big]\big).
\ee
Observe that by (\ref{BDSDE}) and (\ref{Yapprox})
\begin{align} \label{differY}
Y_{\tau_{+}\wedge\bar{\tau}}&-Y_{\tau_{+}\wedge\bar{\tau}}^N= \Phi(\tau,X_{\tau})-\Phi(\bar{\tau},X^N_{\bar{\tau}})\nonumber\\
&+\1_{\{\tau_{+}<\bar{\tau}\}}\bigg(\Int_{\tau_{+}}^{\bar{\tau}}
f(X^N_{\varphi(s)},Y^N_{\varphi(s)},Z^N_{\varphi(s)})ds+
\Int_{\tau_{+}}^{\bar{\tau}}g(X^N_{\psi(s)},Y^N_{\psi(s)},Z^N_{\psi(s)})d\W_s
-\Int_{\tau_{+}}^{\bar{\tau}}Z^N_{s}dB_s\bigg)\nonumber\\
&+\1_{\{\bar{\tau}<\tau_{+}\}}\bigg(\Int_{\bar{\tau}}^{\tau}f(X_s,Y_s,Z_s)ds
+\Int_{\bar{\tau}}^{\tau}g(X_s,Y_s,Z_s)d\W_s-\Int_{\bar{\tau}}^{\tau}Z_sdB_s\bigg).
\end{align}
Then
\begin{align}
\E[|Y_{\tau_{+}\wedge\bar{\tau}}&-Y_{\tau_{+}\wedge\bar{\tau}}^N|^2]\leq \E[|\Phi(\tau,X_{\tau})-\Phi(\bar{\tau},X^N_{\bar{\tau}})|^2]\nonumber\\
&+\E\Big[\1_{\{\tau_{+}<\bar{\tau}\}}\big|\Int_{\tau_{+}}^{\bar{\tau}}f(X^N_{\varphi(s)},Y^N_{\varphi(s)},Z^N_{\varphi(s)})ds\big|^2\Big]+\E\Big[\1_{\{\tau_{+}<\bar{\tau}\}}\big|\Int_{\tau_{+}}^{\bar{\tau}}g(X^N_{\psi(s)},Y^N_{\psi(s)},Z^N_{\psi(s)})d\W_s\big|^2\Big]\nonumber\\
&+\E\Big[\1_{\{\bar{\tau}<\tau_{+}\}}\big|\Int_{\bar{\tau}}^{\tau}f(X_s,Y_s,Z_s)ds\big|^2\Big]+\E\Big[\1_{\{\bar{\tau}<\tau_{+}\}}\big|\Int_{\bar{\tau}}^{\tau}g(X_s,Y_s,Z_s)d\W_s\big|^2\Big].
\end{align} 
We start with the first term in the right hand side of (\ref{differY}). By using \textbf{(MHD)}, \textbf{(MHL)}, \textbf{(MHT)} and Proposition \ref{control} and applying It\^o's lemma to $(\Phi(t,X_t))_{t\geq 0}$ between $\bar{\tau}$ and $\tau$, we easily check that 
\b*
\E[|\Phi(\tau,X_{\tau})-\Phi(\bar{\tau},X^N_{\bar{\tau}})|^2]&\leq &C_L\big(\E[|X_{\bar{\tau}}-X^N_{\bar{\tau}}|^2]+\E\big[\big |\Int_{\bar{\tau}}^{\tau} \Lc \Phi(s,X_s)ds\big|^2\big]\big)\\
&\leq &  C_L\big(\E[|X_{\bar{\tau}}-X^N_{\bar{\tau}}|^2]+\E[\zeta_L|\tau-\bar{\tau}|]\big),
\e*
where $\mathcal L$ is the second order differential operator defined in \eqref{operator}.\\
Then, by appealing to \textbf{(MHD)} and Proposition \ref{control} 
we conclude that 
\be
 \E[|\Phi(\tau,X_{\tau})-\Phi(\bar{\tau},X^N_{\bar{\tau}})|^2]
&\leq & C_L \big(h+ \E[\zeta_L|\tau-\bar{\tau}|]\big).
\ee
For the second term in (\ref{differY}), it follows from Jensen's inequality, the isometry property, the Lipschitz continuous assumption \textbf{(MHL)}, Lemma \ref{lemmcontrol} and Proposition \ref{control} that
\begin{align}
\E\Big[&1_{\{\tau_{+}<\bar{\tau}\}}|\Int_{\tau_{+}}^{\bar{\tau}}f(X^N_{\varphi(s)},Y^N_{\varphi(s)},Z^N_{\varphi(s)})ds|^2\Big]+\E\Big[\1_{\{\tau_{+}<\bar{\tau}\}}|\Int_{\tau_{+}}^{\bar{\tau}}g(X^N_{\psi(s)},Y^N_{\psi(s)},Z^N_{\psi(s)})d\W_s|^2\Big]\nonumber\\
&\leq\E\Big[|\bar{\tau}-\tau_{+}|\Int_{\tau_{+}}^{\bar{\tau}}|f(X^N_{\varphi(s)},Y^N_{\varphi(s)},Z^N_{\varphi(s)})|^2ds\Big]+\E\Big[\Int_{\tau_{+}}^{\bar{\tau}}|g(X^N_{\psi(s)},Y^N_{\psi(s)},Z^N_{\psi(s)})|^2ds\Big]\nonumber\\
&\leq C_L\E\Big[\Int_{\tau_{+}}^{\bar{\tau}}\big(|X^N_{\varphi(s)}|^2 +|Y^N_{\varphi(s)}|^2+\|Z^N_{\varphi(s)}\|^2+|X^N_{\psi(s)}|^2+|Y^N_{\psi(s)}|^2+\|Z^N_{\psi(s)}\|^2\big)ds\Big]\nonumber\\
&\leq C_L\E\Big[\zeta_L(|\bar{\tau}-\tau|+|\tau-\tau_{+}|)\Big]\nonumber\\
&\leq C_L\E[h+\zeta_L|\bar{\tau}-\tau|].
\end{align}
The last term is easily controlled by using the same previous calculations.
\begin{align}
\E\Big[&\1_{\{\bar{\tau}<\tau_{+}\}}\big|\Int_{\bar{\tau}}^{\tau}f(X_s,Y_s,Z_s)ds\big|^2\Big]+\E\Big[1_{\{\bar{\tau}<\tau_{+}\}}|\Int_{\bar{\tau}}^{\tau}g(X_s,Y_s,Z_s)d\W_s|^2\Big]\nonumber\\
&\leq C_L \Big(\E\Big[|\tau-\bar{\tau}|\Int_{\bar{\tau}}^{\tau}\big|f(X_s,Y_s,Z_s)\big|^2ds\Big]+\E\Big[1_{\{\bar{\tau}<\tau_{+}\}}\Int_{\bar{\tau}}^{\tau}|g(X_s,Y_s,Z_s)|^2ds\Big]\Big)\nonumber\\
&\leq C_L \Big(\E\Big[|\tau-\bar{\tau}|\Int_{\bar{\tau}}^{\tau}C_L(|X_s|^2+|Y_s|^2+\|Z_s\|^2)ds\Big]+\E\Big[1_{\{\bar{\tau}<\tau_{+}\}}\Int_{\bar{\tau}}^{\tau}(|X_s|^2+|Y_s|^2+\|Z_s\|^2)ds\Big]\Big)\nonumber\\
&\leq C_L \E\big[|\tau-\bar{\tau}|^2\zeta_L\big]+\E\big[|\tau-\bar{\tau}|\Int_{\bar{\tau}}^{\tau}C_L\|Z_s\|^2ds\big]+C_L\E\Big[\zeta_L|\bar{\tau}-\tau|+1_{\{\bar{\tau}<\tau\}}\Int_{\bar{\tau}}^{\tau}\|Z_s\|^2ds\Big]\nonumber\\
&\leq C_L\E\Big[\zeta_L|\bar{\tau}-\tau|+1_{\{\bar{\tau}<\tau\}}\Int_{\bar{\tau}}^{\tau}\|Z_s\|^2ds\Big].
\end{align}
Finally, we finish the proof of (\ref{differY}) by combining the three estimates.
\ep
\end{proof}

Our next result concerns the regularity of $(Y,Z)$ which was  proved in \cite{B14}:
\begin{Theorem}
Let Assumptions \textbf{(D)}, \textbf{(MHT)}, \textbf{(MHL)} and \textbf{(MHD)} hold. Then
\be \label{regularite}
\Rc(Y)_{\Sc^2}^{\pi}+\Rc(Z)_{\Hc^2}^{\pi}\leq C_L h\quad
\text{and}\quad \E\big[\Int _0^T\|Z_t-\bar{Z}_{\psi(t)}\|^2 dt\big]\leq C_L h.
\ee
\end{Theorem}
Combining the estimates \eqref{proxytau:BGG} and \eqref{regularite}, we finally obtain our main result, which provides an upper bound for the convergence rate of $ \text{Err}(h)_{\tau_{+}\wedge\bar{\tau}}^2$ (and thus for $\text{Err}(h)_{\tau\wedge\bar{\tau}}^2$ and $\text{Err}(h)_{T}^2$).
\begin{Theorem}\label{finalerror}
Let Assumptions \textbf{(D)}, \textbf{(MHT)}, \textbf{(MHL)} and \textbf{(MHD)}  hold. Then, for each $\varepsilon\in(0,1/2)$, there exists $C_L^\varepsilon > 0$ such that
\be 
\text{Err}(h)_{\tau_{+}\wedge\bar{\tau}}^2 \leq C_L h^{1/2} \quad \text{and} \quad \text{Err}(h)_{T}^2 \leq C_Lh^{1/2}
\ee
\end{Theorem}

\section{Semilinear Stochastic PDEs with Dirichlet null condition}
\label{SPDE:section}
The aim of this section is to give a Feynman-Kac's formula for the weak solution of a class of  semilinear SPDEs   with Dirichlet null condition on the boundary
via  the associated Markovian class of BDSDEs  with random terminal time studied in Section \ref{BDSDE:section}. 
Indeed, for a given open connected domain $\Oc$ of $\R^d$,  we are interested in the following semilinear SPDEs  :
\begin{equation}
\left\lbrace 
\label{SPDEdirichlet}
\begin{split}
&du_t+ \Lc  u_t \, dt +  
f(t,x,u_t,D_{\sigma}u_t )\, dt
+ g(t,x,u_t,D_\sigma u_t \,)\, d\W_t = 0 \,,
 \forall \, 0 \leq t \leq T, \; \forall x \in \Oc,\\
&u(T,x) \, = \, \Phi(x) \, ,   \quad \; \; \;  \forall x \in \Oc\\
&u(t,x)= 0 \; , \quad \quad \quad 
 \forall \, 0 \leq t \leq T, \; \forall x \in \partial \Oc,\\
 \end{split}
 \right.
\end{equation}
where  $ D_\sigma \, := \,
\, \nabla \, u \,\sigma  $ and $\mathcal L$ is the second order differential operator which is defined
 component by component with
 \begin{equation}\label{operator}
 \begin{array}{lll}
 {\mathcal L}\varphi(x)&=&\displaystyle\sum_{i=1}^{d}b^i(x)\frac{\partial}{\partial
 x_i}\varphi(x)+\frac{1}{2}\sum_{i,j=1}^{d}a^{ij}(x)\frac{\partial^2}{\partial
 x_i\partial x_j}\varphi(x)
\end{array}
\end{equation}
and $ a := \sigma \sigma^* $.

\subsection{Definitions and formulation}Let us first introduce some notations:\\
- $C^n_{l,b}(\R^p,\R^q)$ is the set of $C^n$-functions which grow at most linearly at infinity and whose partial derivatives of order less than or equal to $n$ are bounded.\\
- $\mathbf{L}^2\left( \Oc\right) $ will be a Hilbert $L^2$-space  of our framework. We employ the following notation for its scalar product and its norm,
$$ \left( u,v\right)=\Int_{\Oc}u\left( x\right) v\left(
x\right) dx,\;\left\| u\right\| _2=\left(
\int_{\Oc}u^2\left( x\right) dx\right) ^{\frac
12}. $$ 
Our evolution problem will be considered over a fixed time interval
$[0,T]$ and the norm for an element of $\mathbf{L}^2\left(
[0,T] \times \Oc\right) $ will be denoted by
$$\left\| u\right\| _{2,2}=\left(\Int_0^T  \int_{\Oc} |u (t,x)|^2 dx dt \right)^{\frac 12}. $$
We assume the following hypotheses :\\[0.3cm]
{\bf{Assumption (MHD')}}
The coefficients of the second  order differential operator $\Lc$ satisfy:
\begin{itemize}
\item $ b$ is a bounded function and belongs to $ C_{l,b}^2 (\R^d, \R^d)$.
\item $ \sigma\in C_{l,b}^3  (\R^d, \R^{k\times d})$ and satisfies the ellipticity condition \eqref{ellipticity}.
\end{itemize}   
{\bf{Assumption (MHT')}}
$\Phi \in \mathbf{L}^2(\Oc; \R^k)$ with polynomial growth, namely there exists $C>0$ and $p\in\N$ such that $|\Phi(x)|\leq C(1+|x|^p)$.\\[0.2cm]
The space of test functions which we employ in the definition of
weak solutions of the evolution equations  \eqref{SPDEdirichlet} is $
\mathcal{D}  := \mathcal{C}^{\infty} (\left[0,T]\right) \otimes
\mathcal{C}_c^{\infty} \left(\Oc\right)$, where
$\mathcal{C}^{\infty} \left([0,T]\right)$ denotes the space of real
functions which can be extended as infinite differentiable functions
in the neighborhood of $[0,T]$ and $
\mathcal{C}_c^{\infty}\left(\Oc\right)$ is the space of
infinite differentiable functions with compact support in
$\Oc$.  We denote by  $ {\mathcal H} $ the space of  $ \Fc_{t,T}^W$-progressively measurable  processes  $(u_t) $ with values  in the Sobolev space $ H_0 ^1 (\Oc)$ where 
$$ H_0^1 (\Oc):=\{v \in \mathbf{L} ^2(\Oc) \; \big|\; \nabla v\sigma\in \mathbf{L}^2(\Oc))\} $$
endowed with the norm
$$\begin{array}{ll}
\|u\|_{{\mathcal H}}^2=
 \E \,  \big[\underset{ 0 \leq s \leq T}{\Sup} \|u_s \|_2^2 +   \Int_{\Oc} \Int_0^T  |\nabla
u_s (x)\sigma(x)|^2 dsdx \big],
\end{array}
$$
where we denote the gradient by $\nabla u (t,x) = \big(\partial_1 u
(t,x), \cdot \cdot \cdot, \partial_d u (t,x) \big)$.

\begin{Definition}
\label{SPDEdefinition}
We say that $u \in  \Hc $ is a weak solution of the SPDE  \eqref{SPDEdirichlet} if 
the following relation holds for each $ \Psi \in \Dc$, 
\begin{equation} \label{variationnelformu}
\begin{split}
& \int_{t}^{T} \int_{\Oc} u (s,x)  \,  
\partial_{s} \Psi  (s,x)   \, dx \,ds
- \int_{\Oc} \Phi(x) \Psi(T,x) dx  +  \int_{\Oc}  u(t,x) \, \Psi(t,x) \; dx   - \int_t ^T  \int_{\Oc}  u (s,x) \Lc^*  u (s,x) \, dx ds  \\
&
=  \int_{t}^{T} \int_{\Oc}\Psi(s,x) \, f(s, x ,u(s,x),D_{\sigma} u(s,x)) \,  dx \, ds +
\int_{t}^{T} \int_{\Oc} \, \Psi (s,x) \, g(s, x ,u(s,x),D_{\sigma} u(s,x))\, dx\, d\W_{s} .
\end{split}
\end{equation}
where 
 $$  \big(u (s, \cdot),  \Lc^*  \Psi  (s,\cdot) \big)  :=   \int_{\Oc}  D_{\sigma} u  (s,x)  \, D_{\sigma}  \Psi (s,x) \,  dx 
+  \int_{\Oc} u (s,x)  \, div ( \, ( b - \tilde{A})\,\Psi (s,x)  ) 
\,  dx , $$ and $\displaystyle  \tilde{A}_{i}  \, =: \, \frac{1}{2} \, \sum_{k=1}^{d}
\frac{\partial a_{k,i}}{\partial x_{k}}.$

\end{Definition}

The existence  and uniqueness  of weak solution for  such SPDEs with null Dirichlet condition is ensured  by Denis and Stoica (Theorem $4$ in \cite{DS04}). Indeed, we can rewrite the second order differential operator $\Lc$  as following:
 \begin{equation}
 \label{operatordivergence}
 \begin{split}
 {\mathcal L} & =  \frac{1}{2}\sum_{i,j=1}^{d}  \partial_{i} \big(a^{ij}(x) \partial_{j} \big) + \displaystyle\sum_{i=1}^{d} \big(b^i(x) -   \frac{1}{2} \partial_{i} a^{ij}(x)  \, \big)\,  \partial_i   .\\
\end{split}
\end{equation}
Therefore,  since $ b$ and $ \nabla a$ are bounded, the second term in the right hand side of   \eqref{operatordivergence} may be considered as an extra term in the nonlinear term coefficient $f$ which still satisfies the uniform Lipschitz continuous condition in $ u $ and $ D_{\sigma} u$.

 Motivated by developing  Euler numerical  scheme for  such solution,  we are  now interested  in giving the probabilistic interpretation for the solution of SPDEs \eqref{SPDEdirichlet}   within the framework of BDSDE with random terminal time. Thus, this connection between SPDEs and BDSDEs will be established by means of  stochastic flow technics. 
 \subsection{Stochastic flow of diffeomorphism and random test functions}
\label{Flow:subsection} We are concerned in this paper with solving SPDEs by developing a stochastic flow method which was first introduced in Kunita \cite{K84}, and Bally, Matoussi
\cite{BM01}. We recall that  $\{X_s^{t,x}, t\leq s\leq T\}$ is the diffusion process starting from $x$ at time $t$ and is the strong solution  of the equation:
 \begin{equation}\label{sde}
X_s^{t,x}=x+\Int_{t}^{s}b(X_r^{t,x})dr+\Int_{t}^{s}\sigma(X_r^{t,x})dB_r.
\end{equation}
The existence and uniqueness of this solution was proved in Kunita \cite{K84}. Moreover, we have the following properties:
\begin{Proposition}\label{estimatesde}
For each $t>0$, there exists a version of $\{X_s^{t,x});\,x\in
\R^d,\,s\geq t\}$ such that $X_s^{t,\cdot}$ is a $C^2(\R^d)$-valued
continuous process which satisfies the flow property: $ X_r^{t,x}= X_r^{s,x} \circ  X_s^{t,x}$, $0\leq t<s<r$.
Furthermore,
for all $p\geq 2$, there exists $M_p$ such that for all $0\leq t<s$,
$x,x'\in\R^d$, $h,h'\in\R\backslash{\{0\}}$,
$$\begin{array}{ll}
\E(\underset{t\leq r\leq s}{\Sup}|X_r^{t,x}-x|^p)\leq
M_p(s-t)(1+|x|^p),\\
\E(\underset{t\leq r\leq s}{\Sup}|X_r^{t,x}-X_r^{t,x'}-(x-x')|^p)\leq
M_p(s-t)(|x-x'|^p),\\
\E(\underset{t\leq r\leq s}{\Sup}|\Delta_h^i[X_r^{t,x}-x]|^p)\leq
M_p(s-t),\\
\E(\underset{t\leq r\leq s}{\Sup}|\Delta_h^i X_r^{t,x}-\Delta_{h'}^i
X_r^{t,x'}|^p)\leq M_p(s-t)(|x-x'|^p+|h-h'|^p),
\end{array}$$
where $\Delta_h^ig(x)=\frac{1}{h}(g(x+he_i)-g(x))$, and
$(e_1,\cdots,e_d)$ is an orthonormal basis of $\R^d$.
\end{Proposition}
Under regular conditions Assumption {\bf{(MHD')}}  on the diffusion, it is known that the stochastic flow associated to a continuous SDE satisfies the homeomorphic property (see Kunita \cite{K84}). We have  the following result
where the proof can be found in  \cite{K84}.
\begin{Proposition}\label{flow}
Let  Assumption {\bf{(MHD')}} holds. Then
$\{X_s^{t,x}; x \in \mathbb{R}^d \}$ is a $C^2$-diffeomorphism a.s.
stochastic flow. Moreover the inverse of the flow which denoted by  $ \{X_{t,s}^{-1} (y); y \in \mathbb{R}^d \} $  satisfies the
following backward SDE
\begin{equation}\label{inverse:flow}
\begin{split}
 X_{t,s}^{-1}(y) &  = y - \Int_t^s \widehat{b}(X_{r,s}^{-1}(y)) dr  -
\Int_t^s \sigma (X_{r,s}^{-1} (y)) d\B_r
\end{split}
\end{equation}
for any  $t<s$,  where
\begin{equation}
\label{drift:backward}
\begin{split}
\widehat{b}(x) =  b(x) -  \sum_{i,j}\frac{\partial \sigma^j (x)
}{\partial x_i}  \sigma^{ij} (x) .
\end{split}
\end{equation}
\end{Proposition}
We denote by $J(X_{t,s}^{-1}(x))$ the determinant of the Jacobian
matrix of $X_{t,s}^{-1}(x)$, which is positive and
$J(X_{t,t}^{-1}(x))=1$. For $\varphi\in C_c^{\infty}(\R^d)$, we define
a process $\varphi_t:\,\Omega\times [t,T]\times \R^d\rightarrow \R^k$ by
\begin{equation}
\label{random:testfunction}
\varphi_t(s,x):=\varphi(X_{t,s}^{-1}(x))J(X_{t,s}^{-1}(x)).
\end{equation}
We know that for $v\in \mathbf{L}^2(\R^d)$, the composition of $v$
with the stochastic flow is
$$(v\circ X_s^{t, \cdot } , \varphi):=(v,\varphi_t(s,\cdot)).$$
In fact, by a change of variable, we have (see Kunita \cite{K94a}, Bally and Matoussi \cite{BM01}) $$(v\circ X_s^{t, \cdot },
\varphi)=\Int_{\R^d}v(X_s^{t, x})\varphi(x)dx=\Int_{\R^d}v(y)\varphi(X_{t,s}^{-1}(y))J(X_{t,s}^{-1}(y))dy
=(v,\varphi_t(s,\cdot)).$$ Since  $ (\varphi_t(s,x))_{ t\leq s}$ is a
process,  we may not use it directly as a test function because \\
$\Int_t^T(u(s,\cdot),\partial_s\varphi_t(s,\cdot))ds$ has no sense. However
$\varphi_t(s,x)$ is a semimartingale and we have the following
decomposition of $\varphi_t(s,x)$
\begin{Lemma}\label{decomposition}
For every function $\varphi\in C_c^{\infty}(\R^d),$
\begin{equation}\label{decomp}\begin{array}{ll}
\varphi_t(s,x)&=\varphi(x)+\displaystyle\int_t^s{\mathcal
L}^\ast\varphi_t(r,x)dr-\sum_{j=1}^{d}\int_t^s\left(\sum_{i=1}^{d}\frac{\partial}{\partial
x_i}(\sigma^{ij}(x)\varphi_t(r,x))\right)dW_r^j,
 \end{array} 
 \end{equation}
where ${\mathcal L}^\ast$ is the adjoint operator of ${\mathcal L}$.
\end{Lemma}

We also  need  equivalence of norms result  which plays an important role in the proof of the existence of the solution for SPDE as a connection between the functional norms and random norms.
For continuous SDEs, this result  was first proved by Barles and Lesigne \cite{BL97} by
using an analytic method and Bally and Matoussi \cite{BM01} by a probabilistic method. 
\begin{Proposition}\label{equivalence:normes}
There exists two constants $c>0$ and $C>0$ such that for every
$t\leq s\leq T$ and $\varphi\in L^1(\R^d)$,
\begin{equation}\label{equi1} c\Int_{\R^d}|\varphi(x)|dx\leq
\Int_{\R^d} \E(|\varphi(X_s^{t, x})|) dx\leq
C\Int_{\R^d}|\varphi(x)| dx. 
\end{equation} 
Moreover, for
every $\Psi\in L^1([0,T]\times\R^d)$,
\begin{equation} \label{equi2}
c\Int_{\R^d}\Int_t^T|\Psi(s,x)|ds dx \leq
\Int_{\R^d}\Int_t^T \E(|\Psi(s,X_s^{t, x })|)ds dx\leq
C\Int_{\R^d}\Int_t^T|\Psi(s,x)|ds dx.
\end{equation}
\end{Proposition}

 We give now the following result which allows us to link by a natural way the solution of SPDE with the associated BDSDE.
 Roughly speaking, if we choose in the variational formulation \eqref{variationnelformu} the random functions $\varphi_t(\cdot,\cdot)$
 defined by \eqref{random:testfunction}, as a test functions, then we obtain the  associated BDSDE.
 In fact, this result plays the same role as It\^o's formula used in \cite{pp1994} 
 to relate the solution of some semilinear SPDEs with the associated BDSDEs:
\begin{Proposition}
\label{weak:Itoformula1} Let  Assumptions  {\bf{(MHT')}},  {\bf{(MHL)}} and  {\bf{(MHD')}} hold and $u\in {\mathcal H}$ be a weak solution of the
SPDE \eqref{variationnelformu} associated to $(\Phi,f,g)$ on the whole domain $ \mathbb R^d$, 
then for $s\in[t,T]$ and $\varphi \in
C_c^{\infty}(\R^d)$, 
\begin{equation} \label{weakvariationnelformu}
\begin{split}
&\Int_{\R^d}\int_s^T\!u(r,x)d\varphi_t(r,x)dx+(u(s,\cdot),\varphi_t(s,\cdot))-(\Phi(\cdot),\varphi_t(T,\cdot))
-\Int_{\R^d}\int_s^T\! u(r,x)\mathcal L^\ast \varphi_t(r,x))drdx\\
&=\Int_{\R^d}\int_s^Tf_r(x,u(r,x),
D_{\sigma} u(r,x))\varphi_t(r,x)drdx+\Int_{\R^d}\Int_s^T\! g_r(x,u(r,x), D_{\sigma} u(r,x)\sigma(x))\varphi_t(r,x)d\W_rdx,  
\end{split}
\end{equation}
where  $ \Int_{\R^d}\int_s^Tu(r,x)d\varphi_t(r,x)dx $ is well defined
thanks to the semimartingale decomposition result (Lemma
\ref{decomposition}).
\end{Proposition}

\subsection{Probabilistic representation of the solution of SPDE}
\label{main/section}
As introduced in Section  \ref{Numeric scheme:section},  we consider  now the  Markovian BDSDE with random terminal time $\tau^{t,x}$  which is the first exist time of the  forward diffusion $X^{t,x}$ from the domain $\Oc$
\begin{equation}
\label{MBDSDE}
\begin{split}
 Y_{s}^{t,x} \; =& \; 
 \Phi(X_{T\wedge\tau^{t,x}}^{t,x})\;  + \; 
 \Int_{s}^{T} \1_{\left(\tau^{t,x}> r\right)}
f(r,X_{r}^{t,x},Y_{r}^{t,x},Z_{r}^{t,x})\, dr  - \Int_{s}^{T}  Z_{r}^{t,x,} \,  dB_{r}  \\
& \quad + \quad 
 \Int_{s}^{T} \1_{\left(\tau^{t,x}> r\right)}
 g(r,X_{r}^{t,x},Y_{r}^{t,x},Z_{r}^{t,x})
\,d\W_r\, .
\end{split}
\end{equation}
\begin{Remark}
We have $ Y_{s}^{t,x}\, =\, Z_{s}^{t,x}\, =\,0, 
\; \forall \,  \tau^{t,x} \, \leq s \, \leq T$. 
In fact, the process $ Z^{t,x} $ is the density which appears in 
the Ito's representation theorem of the random variable 
$$
\xi = \Phi(X_{T\wedge\tau^{t,x}}^{t,x}) \, + \,  
\Int_{s}^{T} \, \1_{(\tau^{t,x} \, > \, r \, ) } \,
 f(r, X_r^{t,x} ,Y_r^{t,x},Z_r^{t,x}) \,  dr
$$
But, the r.v $ \xi $ is $ \Fc_{\tau^{t,x}}^W$-measurable, then 
$Z_{r}^{t,x}= Z_{r}^{t,x}\,\1_{(\tau^{t,x} \, \geq  \, r \, ) }$. Now, we look 
at \eqref{MBDSDE} for $ T \, \geq s \, > \, \tau^{t,x}$,  all the terms in the
right hand of \eqref{MBDSDE} vanisch, then $ Y_{s}^{t,x}$ vanischs, for
$ T \, \geq s \, > \, \tau^{t,x}$.
\end{Remark}
The main result in  this section is the following 
\begin{Theorem}
\label{existence:SPDEs}
Assume  {\bf{(MHT')}}, {\bf{(D)}}, {\bf{(MHL)}} and  {\bf{(MHD')}} hold  and let $ \{ (Y_{s}^{t,x}, \, Z_{s}^{t,x}), t \leq s
\leq T \} $ be the solution of BDSDE   \eqref{MBDSDE} . Then, $ u(t,x) := Y_{t}^{t,x},$ $ dt\otimes  dx, \; a.e.$  is the unique solution
of  the SPDE \eqref{variationnelformu} 
and 
\begin{equation}
\label{representation} 
 Y_{s}^{t,x} \, = \, u (s\wedge\tau^{t,x},X_{s\wedge\tau^{t,x}}^{t,x}), \quad 
 Z_{s}^{t,x} \, = \,  D_{\sigma} u  (s\wedge\tau^{t,x},X_{s\wedge\tau^{t,x}}^{t,x}).
\end{equation} 
\end{Theorem}

\begin{proof}
\label{proof theorem:subsection}
\textit{Step 1: local variational form of SPDE }  \\[0.2cm]
  Let $u  \in \mathcal H $ be  a weak solution  of \eqref{SPDEdirichlet} and let  $ \theta \, \in \, C_{c}^{1} (\Oc) $.  Then, we apply  the variational equation \eqref{variationnelformu}  for the test 
function $ \theta \, \Psi $, with $ \Psi \, \in \mathcal{C}^{\infty} (\left[0,T]\right) \otimes
\mathcal{C}_c^{\infty} \left(\Oc\right)$ to obtain  
\begin{equation} \label{variationnel}
\begin{split}
& \int_{t}^{T} \int_{\Oc} u (s,x)  \, \theta(x)\, 
\partial_{s} \Psi  (s,x)   \, dx \,ds
- \int_{\Oc} \Phi(x)\theta(x) \Psi(T,x) dx  +  \int_{\Oc}  u(t,x)\,\theta(x) \, \Psi(t,x) \; dx \\ 
& 
- \int_{t}^{T}\int_{\Oc}  D_{\sigma} u  (s,x)  \,\theta(x)\, D_{\sigma}  \Psi (s,x) \,  dx \, ds
-  \int_{t}^{T}\int_{\Oc} u (s,x)  \, div ( \, ( b - \tilde{A})\,\theta(x)\,\Psi (s,x)  ) 
\,  dx \, ds  \\
&  
=  \int_{t}^{T} \int_{\Oc}\Psi(s,x)\big[\theta(x) \, f(s, x ,u(s,x),D_{\sigma} u(s,x)) +D_{\sigma} u(s,x)D_{\sigma} \theta(x)\,\big]  dx \, ds \\
&+
\int_{t}^{T} \int_{\Oc} \,\theta(x)\, \Psi (s,x) \, g(s, x ,u(s,x),D_{\sigma} u(s,x))\, dx\, d\W_s .
\end{split}
\end{equation}
Since $\theta$ has a compact support on $\Oc$, we can rewrite the variational formulation \eqref{variationnel} in the whole domain $\R^d$
\begin{equation} \label{variationneldom}
\begin{split}
& \int_{t}^{T} \int_{\R^d} u (s,x)  \, \theta(x)\, 
\partial_{s} \Psi  (s,x)   \, dx \,ds
- \int_{\R^d} \Phi(x)\theta(x) \Psi(T,x) dx  +  \int_{\R^d}  u(t,x)\,\theta(x) \, \Psi(t,x) \; dx \\ 
& 
- \int_{t}^{T}\int_{\R^d}  D_{\sigma} u  (s,x)  \,\theta(x)\, D_{\sigma}  \Psi (s,x) \,  dx \, ds
-  \int_{t}^{T}\int_{\R^d} u (s,x)  \, div ( \, ( b - \tilde{A})\,\theta(x)\,\Psi (s,x)  ) 
\,  dx \, ds  \\
&  
=  \int_{t}^{T} \int_{\R^d}\Psi(s,x)\big[\theta(x) \, f(s, x ,u(s,x),D_{\sigma} u(s,x)) +D_{\sigma} u(s,x)D_{\sigma} \theta(x)\,\big]  dx \, ds \\
&+
\int_{t}^{T} \int_{\R^d} \,\theta(x)\, \Psi (s,x) \, g(s, x ,u(s,x),D_{\sigma} u(s,x))\, dx\, d\W_s .
\end{split}
\end{equation}
Then, from Proposition \ref{weak:Itoformula1}, which gives the weak variational formulation \eqref{variationneldom} applied to random test function $ \varphi_t(\cdot, \cdot)$ 
 \eqref{random:testfunction} yields to:

\begin{equation} \label{variationel}
\begin{split}
& \int_{s}^{T} \int_{\R^d} u (r,x)  \, \theta(x)\, 
d_{r} \varphi_t  (r,x)   \, dx \,dr
- \int_{\R^d} \Phi(x)\theta(x) \varphi_t(T,x) dx  +  \int_{\R^d}  u(s,x)\,\theta(x) \, \varphi_t(s,x) \; dx \\ 
& 
-  \int_{s}^{T}\int_{\R^d}  D_{\sigma} u  (r,x)  \,\theta(x)\, D_{\sigma}  \varphi_t (r,x) \,  dx \, dr -  \int_{s}^{T}\int_{\R^d} u (r,x)  \, div ( \, ( b - \tilde{A})\,\theta(x)\,\varphi_t (r,x)  ) 
\,  dx \, dr  \\
&  
=  \int_{s}^{T} \int_{\R^d}\varphi_t(r,x)\big[\theta(x) \, f(r, x ,u(r,x),D_{\sigma} u(r,x)) +D_{\sigma} u(r,x)D_{\sigma} \theta(x)\,\big]  dx \, dr \\
&+
\int_{s}^{T} \int_{\R^d} \,\theta(x)\, \varphi_t (r,x) \, g(r, x ,u(r,x),D_{\sigma} u(r,x))\, dx\, d\W_r.
\end{split}
\end{equation}
Moreover,  by Lemma \ref{decomposition}, we have that 
\begin{equation*} 
\begin{split}
\int_{s}^{T} \int_{\R^d} u (r,x)  \, \theta(x)\, 
d_{r} \varphi_t  (r,x)   \, dx \,dr=& \int_{\R^d} \int_s^T u(r,x) \, \theta(x) \, \Lc^* \varphi_{t} (r,x) \, dr \, dx\\ 
& - \int_{\R^d} \int_s^T u(r,x) \, \theta(x) \, 
 \nabla \, 
\left(\sigma^{*}(x)\,\varphi_{t} (r,x)\right)(x)\, dB_r \, dx.
\end{split}
\end{equation*}
Using Integration by parts, we obtain
\begin{equation*}
\begin{split}
&\Int_{s}^{T}\!\int_{\R^d}\! u (r,x)\theta(x) 
d_{r} \varphi_t  (r,x)  dx \,dr  =
\Int_s^T\!\Int_{\R^d}\! D_\sigma \left(u(r,x) \theta(x) \right)
\varphi_{t} (r,x) dx dB_r \\
&+  \Int_s^T \!\Int_{\R^d}\!  D_\sigma \left( u(r,x) \theta(x) \right)
D_\sigma \varphi_{t} (r,x) dx dr+ \int_{s}^{T}\int_{\R^d} u (r,x)  \, div ( \, ( b - \tilde{A})\,\theta(x)\,\varphi_t (r,x)  ) 
\,  dx \, dr  \\
& =  \Int_s^T\! \Int_{\R^d}\!\theta(x)\left( D_\sigma  u(r,x)  \right)
\varphi_{t} (r,x) dx dB_r  + 
 \Int_s^T\! \Int_{\R^d}\! u(r,x) D_\sigma  \theta(x)
\varphi_{t} (r,x) dx  dB_r \\
& +  \Int_s^T\! \Int_{\R^d}\! \theta(x)D_\sigma  u(r,x) 
D_\sigma \varphi_{t} (r,x) dx dr  +
\Int_s^T \Int_{\R^d}  u(r,x) D_\sigma  \theta(x)
D_\sigma \varphi_{t} (r,x)dx  dr\\
& + \int_{s}^{T}\int_{\R^d} u (r,x)  \, div ( \, ( b - \tilde{A})\,\theta(x)\,\varphi_t (r,x)  ) 
\,  dx \, dr .
\end{split}
\end{equation*}
Using again integration by parts for the fourth term 
in the right hand of the above equation, we get 
\begin{equation}
\begin{split}
&\Int_{s}^{T}\! \Int_{\R^d}\! u (r,x) \theta(x) 
d_{r} \varphi_t  (r,x)   dx \,dr  =
\Int_s^T\! \Int_{\R^d}\!\theta(x)\left(D_\sigma u(r,x) \right)
\varphi_{t} (r,x) dx dB_r\\ 
&+
\Int_s^T\! \Int_{\R^d} u(r,x) D_\sigma \theta(x) 
\varphi_{t} (r,x) dx dB_r
 +  \Int_s^T \Int_{\R^d} \theta(x)D_\sigma  u(r,x) 
D_\sigma \varphi_{t} (r,x) dx dr \\
&+ \int_{s}^{T}\int_{\R^d} u (r,x)  \, div ( \, ( b - \tilde{A})\,\theta(x)\,\varphi_t (r,x)  ) 
\,  dx \, dr \\
&- 
\Int_s^T\! \Int_{\R^d}\!  \big(D_\sigma u(r,x) D_\sigma  \theta(x)+u(r,x) D^2_\sigma  \theta(x)\big)
\varphi_{t} (r,x)dx  dr. 
\end{split}
\end{equation}
We substitute now the above equation in \eqref{variationel} to get 
\begin{equation*}
\begin{split}
& 
\int_{\R^d}  u(s,x)\,\theta(x) \, \varphi_t(s,x) \; dx - \int_{\R^d} \Phi(x)\theta(x) \varphi_t(T,x) dx   \\ 
& 
+ \Int_s^T\! \Int_{\R^d}\!\theta(x)\left(D_\sigma u(r,x) \right)
\varphi_{t} (r,x) dx dB_r +
\Int_s^T\! \Int_{\R^d} u(r,x) D_\sigma \theta(x) 
\varphi_{t} (r,x) dx dB_r  \\
&  
=  \int_{s}^{T} \int_{\R^d}\varphi_t(r,x)\big[\theta(x) \, f(r, x ,u(r,x),D_{\sigma} u(r,x)) -u(r,x) D^2_\sigma  \theta(x)\,\big]  dx \, dr \\
&+
\int_{s}^{T} \int_{\R^d} \,\theta(x)\, \varphi_t (r,x) \, g(r, x ,u(r,x),D_{\sigma} u(r,x))\, dx\, d\W_r.
\end{split}
\end{equation*}
Now the change of variable $ y = X^{-1}_{t,s} (x) $ in the above equation gives 
\begin{equation*}
\begin{split}
& 
\int_{\R^d}  u(s,X_s^{t,x})\,\theta(X_s^{t,x}) \, \varphi(x) \; dx - \int_{\R^d} \Phi(X_T^{t,x})\theta(X_T^{t,x}) \varphi(x) dx  \\ 
& 
+ \Int_s^T\! \Int_{\R^d}\!\theta(X_r^{t,x})\left(D_\sigma u(r,X_r^{t,x}) \right)
\varphi(x) dx dB_r + 
\Int_s^T\! \Int_{\R^d} u(r,X_r^{t,x}) D_\sigma \theta(X_r^{t,x}) 
\varphi(x) dx dB_r  \\
&  
=  \int_{s}^{T} \int_{\R^d}\varphi(x)\big[\theta(X_r^{t,x}) \, f(r, X_r^{t,x} ,u(r,X_r^{t,x}),D_{\sigma} u(r,X_r^{t,x})) -u(r,X_r^{t,x}) D^2_\sigma  \theta(X_r^{t,x})\,\big]  dx \, dr \\
&+
\int_{s}^{T} \int_{\R^d} \,\theta(X_r^{t,x})\, \varphi (X_r^{t,x}) \, g(r, X_r^{t,x} ,u(r,X_r^{t,x}),D_{\sigma} u(r,X_r^{t,x}))\, dx\, d\W_{r}.
\end{split}
\end{equation*}
Define
$ Y_s^{t,x} := u(s, X_s^{t,x}),$ \,a.e.\, and $ Z_s^{t,x} := D_\sigma u(s, X_s^{t,x})\,  $a.e.. In particular we have $ u(t, x)=Y_t^{t,x} ,  $\,a.e.\, and $ D_\sigma u(t,x) =Z_t^{t,x}, $\, a.e.. Thus, it follows from the last equation
\begin{equation*}
\begin{split}
& 
\int_{\R^d} \big[Y_s^{t,x}\,\theta(X_s^{t,x})  -  Y_T^{t,x}\theta(X_T^{t,x}) \big]\varphi(x) dx    \\
&  
=  \int_{\R^d}  \int_{s}^{T} \big[\theta(X_r^{t,x}) \, f(r, X_r^{t,x} , Y_r^{t,x},Z_r^{t,x}) -u(r,X_r^{t,x}) D^2_\sigma  \theta(X_r^{t,x})\,\big]   \varphi(x)\ \, dx \, dr \\
&+
\int_{\R^d}  \int_{s}^{T}\! \!\theta(X_r^{t,x}) \, g(r, X_r^{t,x} , Y_r^{t,x},Z_r^{t,x})\, \varphi (x) \,  d\W_{r} \,  dx \\
&-  \Int_{\R^d}  \Int_s^T    \big[  \theta(X_r^{t,x})Z_r^{t,x} -  Y_r^{t,x} D_{\sigma} \theta (X_r^{t,x}) \,  \big] 
dB_r \,  \varphi(x) dx .
\end{split}
\end{equation*}

Since $ \varphi \in C_c^{\infty} (\mathbb R^d)$ is arbitrary function, we get  the following  equation

\begin{equation}
\label{BSDE:truncated}
\begin{split}
& 
Y_s^{t,x}\,\theta(X_s^{t,x})  =  Y_T^{t,x}\theta(X_T^{t,x})     +   \int_{s}^{T} \big[\theta(X_r^{t,x}) \, f(r, X_r^{t,x} , Y_r^{t,x},Z_r^{t,x}) -u(r,X_r^{t,x}) D^2_\sigma  \theta(X_r^{t,x})\,\big]  \, dr \\
& \int_{s}^{T}\! \!\theta(X_r^{t,x}) \, g(r, X_r^{t,x} , Y_r^{t,x},Z_r^{t,x}) \,  d\W_{r}  -   \Int_s^T    \big[  \theta(X_r^{t,x})Z_r^{t,x} -  Y_r^{t,x} D_{\sigma} \theta (X_r^{t,x}) \,  \big] 
dB_r.
\end{split}
\end{equation}
\textit{Step 2: Approximation of the random terminal time and BDSDE }  \\[0.2cm]

We denote the set $\Oc_\epsilon$ by  $\Oc_\epsilon := \{ x \, \in \Oc \; : \, 
d \, ( x, \Oc^c \, ) \; > \, \epsilon \, \} $ and the function 
$$
\theta_\epsilon (x) := 
\left\lbrace 
\begin{array}{l}
1 \, , \quad \quad \quad  x \, \in \,  \Oc_\epsilon ,\\
0 \, , \quad \quad \quad  x \, \in \,  \Oc_{\frac{\epsilon}{2}}^c.
\end{array}
\right.
$$
So, $ 0 \leq \theta_\epsilon (x) \, \leq \, 1 $ and 
$\theta_\epsilon \, \in \, C_c^\infty ( \Oc_\epsilon)$.
 We define the exit stoping time from the set $\Oc_\epsilon$ by 
$$
\tau_\epsilon^{t,x} := 
\inf \{ \, t < s \leq T \; : \;  X_s^{t,x} \, \notin \,  \Oc_\epsilon
\; \} \wedge (T-\varepsilon(T-t))\in[t,T]
.
$$
Then, for $ t \, \leq s \, \leq \, \tau_\epsilon^{t,x}$, we have
$\theta_\epsilon (X_s^{t,x}) = 1 $ and 
$ D_\sigma \theta_\epsilon (X_s^{t,x}) =
 D_\sigma^2 \theta_\epsilon (X_s^{t,x})=0$. Then, we use the localization function $ \theta_\epsilon $ in Equation  \eqref{BSDE:truncated} to get
\begin{equation}
\label{BDSDE:tauepsilon}
\begin{split}
Y_{s\wedge \tau_\epsilon^{t,x}}^{t,x}  &  = Y_{ \tau_\epsilon^{t,x}}^{t,x}  + 
 \int_{s\wedge \tau_\epsilon^{t,x}}^{\tau_\epsilon^{t,x}}
 f(r, X_r^{t,x} ,Y_r^{t,x},Z_r^{t,x}) \,  dr \\
&  \quad \quad  + 
 \int_{s\wedge \tau_\epsilon^{t,x}}^{\tau_\epsilon^{t,x} }   
 g(r, X_r^{t,x} ,Y_r^{t,x},Z_r^{t,x}) \, d\W_{r}  - \int_{s\wedge \tau_\epsilon^{t,x}}^{\tau_\epsilon^{t,x}}
  Z_r^{t,x}  \, dB_r  .
\end{split}
\end{equation}

Since the domain $\Oc$ is smooth enough satisfying Assumption {\bf{D}}, we have that the stoping time $\tau_\epsilon^{t,x}$ converge 
to the stoping time $\tau^{t,x}$ a.s, where 
$\tau^{t,x}:= \inf \{ \, t < s \; : \;  X_s^{t,x} \, \notin \,  \Oc
\; \}\wedge T$  (see  Chapter IV page 119-120 in Gobet \cite{G13}) .\\
So, passing to the limit  in the BDSDE  \eqref{BDSDE:tauepsilon}, we obtain
\begin{equation}
\label{BDSDEbis}
\begin{split}
Y_{s\wedge \tau^{t,x}}^{t,x}  &  = Y_{ \tau^{t,x}}^{t,x}  + 
 \int_{s\wedge \tau^{t,x}}^{\tau^{t,x}}
 f(r, X_r^{t,x} ,Y_r^{t,x},Z_r^{t,x}) \,  dr \\
&  \quad \quad  + 
 \int_{s\wedge \tau^{t,x}}^{\tau^{t,x} }   
 g(r, X_r^{t,x} ,Y_r^{t,x},Z_r^{t,x}) \, d\W_{r}  - \int_{s\wedge \tau^{t,x}}^{\tau^{t,x}}
  Z_r^{t,x}  \, dB_r  .
\end{split}
\end{equation}In the other hand, $ Y_{T\wedge \tau^{t,x}}^{t,x} = \Phi(X_{T\wedge \tau^{t,x}}^{t,x})  $. Indeed, using the boundary condition
 of the solution $u$ of the SPDE, we get  
$ Y_{T\wedge \tau^{t,x}}^{t,x} = u \left(\tau^{t,x}, \,X_{\tau^{t,x}}
^{t,x} \, \right)=0 $ which complete the proof of Theorem \ref{existence:SPDEs} and in particular the representation \eqref{representation}. 
\ep \\
\end{proof}
\begin{Remark}
We may get uniqueness of the solution for the SPDE \eqref{SPDEdirichlet}  from the probabilistic representation. Indeed,  let $u$ and $ \bar{u} $ to be two solutions of the SPDE \eqref{SPDEdirichlet} and 
$ ( Y, \, Z ) $ and $ ( \bar{Y}, \, \bar{Z} ) $ are the two associated 
solutions of the BDSDEs \eqref{BDSDEbis}. We denote by $ \Delta u := u- \bar{u} $,
$\Delta Y := Y- \bar{Y} $ and $\Delta Z := Z- \bar{Z} $.  By usual computations on BSDEs, we obtain that 
 $ \Delta u (t,x) = \,\Delta Y_{t\wedge \tau^{t,x}}^{t,x} = 0 , \; \forall
\, x \, \in \, \Oc.$ So, the uniqueness of the solution of the SPDE is given by the 
uniqueness of the BDSDEs.
\end{Remark}
\subsection{Numerical Scheme for SPDE}
Let us first recall that $ (X^N, Y^N, Z^N)$  denotes  the numerical  Euler scheme of the FBDSDEs \eqref{forward}-\eqref{BDSDE}   given in  \eqref{euldiscrete}-\eqref{conterminale}-\eqref{Yn}-\eqref{Zn}.  The numerical approximation of the SPDE  \eqref{SPDEdirichlet} will  be presented in the following lemma:
\begin{Lemma}\label{num-Scheme-SPDE}
Let $x\in \Oc$ and $t_n\in \pi$. Define
\begin{eqnarray}\label{numspde}
u^{N}(t_{n},x):=Y_{t_{n}}^{N,t_{n},x} \textrm{ and } v^{N}(t_{n},x):=Z_{t_{n}}^{N,t_{n},x}
\end{eqnarray}
Then $u^{N}(t_n,\cdot)$ (resp. $v^{N}(t_n,\cdot)$) is $\mathcal{F}_{t_{n},T}^{B}$-measurable and we have for all $x\in \Oc$ and $t,t_n\in \pi$ such that $t\leq t_n$ \\
\begin{eqnarray*}
u^{N}(t_n,X_{t_{n}}^{t,x})=Y_{t_{n}}^{N,t,x}\quad (\textrm{resp. }v^{N}(t_n,X_{t_{n}}^{t,x})=Z_{t_{n}}^{N,t,x}).
\end{eqnarray*}
\end{Lemma}
We define the error between the solution of the SPDE and the numerical scheme as follows:
\begin{eqnarray}\label{Error(u,v)}
Error_{N}(u,v)&:=& \Sup_{0\leq s \leq T}\E\big[\Int_{\Oc}|u_{s}^{N}(x)-u(s, x)|^{2}\rho(x)dx\big]\nonumber\\
&+&\sum_{n=0}^{N-1}\E\big[\Int_{\Oc}\Int_{t_{n}}^{t_{n+1}}\|v_{s}^{N}(x)-v(s,x)\|^{2}ds\rho(x)dx\big].
\end{eqnarray}
The following theorem shows the convergence of the numerical scheme \reff{numspde} of the solution of the SPDE \reff{SPDEdirichlet}.
\begin{Theorem}
Assume  {\bf{(MHT')}}, {\bf{(D)}}, {\bf{(MHL)}} and  {\bf{(MHD')}} hold. Then, the error $Error_{N}(u,v)$
converges to 0 as $N\rightarrow\infty$ and there exists a positive constant $C_L$ such that
\begin{eqnarray}\label{ratecv-uv}
Error_{N}(u,v)\leq C_L h.
\end{eqnarray}
\end{Theorem}
We can follow the same arguments presented in \cite{matouetal13} (see Theorem 5.2). So, the proof is omitted.

\section{Implementation and numerical tests}
In this part, we are interested in implementing our numerical scheme. Our aim is only to test statically its convergence. Further analysis of the convergence of the used method and of the error bounds will be accomplished in a future work. All the numerical tests have been performed on a PC equipped with a processor Intel Core i7 (dual core) with 2.80 Ghz with codes written in C and compiled with GCC (GNU).
\subsection{Notations and algorithm}
Not forgetting that $\pi:=\{t_i=ih\;;\; i\leq N\},\, h:= T/N\;, \; N\in\N,$ is the time grid of the interval $[0,T]$.
We use a path-dependent algorithm, for every fixed path of the brownian motion $B$, we approximate by a regression method the solution of the associated PDE. Then, we replace the conditional expectations which appear in (\ref{2}) and (\ref{3}) by $\mathbb{L}^{2}(\Omega,\mathcal{P})$ projections on the function basis approximating $\mathbb{L}^{2}(\Omega,\mathcal{F}_{t_{n}})$. We compute $Z_{t_{n}}^{N}$ in an explicit manner  and we use I Picard iterations to compute $Y_{t_{n}}^{N}$ in an implicit way. Actually, we proceed as in \cite{GLW05}, except that in our case the solutions $Y_{t_{n}}^{N}$ and $Z_{t_{n}}^{N}$  are measurable functions of $(X^{N}_{t_{n}},(\Delta B_{i})_{n\leq i \leq N-1} )$. So, each solution given by our algorithm depends on the fixed path of B.
\subsubsection{Forward Euler scheme}
\label{Forward Euler scheme:subsection}
The discrete approximation of the forward diffusion process (\ref{forward}) is defined by

\begin{equation}\label{eulerdiscrete}
\left\{
\begin{array}{ll}
X_0^N=x,\\
X_{t_{i+1}}^N= X_{t_{i}}^N+b(X_{t_{i}}^N)(t_{i+1}-t_{i})+\sigma(X_{t_{i}}^N)(B_{t_{i+1}}-B_{t_{i}}),\quad i\leq N.
\end{array}
\right.
\end{equation}
Then, we approximate the exit time $\tau$ by the first time of the Euler scheme $(t,X_t^N)_{t\in\pi}$ from $D$ on the grid $\pi$:
\b*
\bar{\tau}:=\Inf\{t\in\pi : X_t^N\notin\Oc\}\wedge T.
\e*
The simulation of the diffusion stopped at the exit time is based on the approach of Gobet and Menozzi \cite{GM10}. In this approach, we simulate the diffusion with an Euler scheme with step size $h$ and stop it at discrete times $(t_i)_{i\in\N^*}$ in a modified domain, whose boundary has been appropriately shifted. The shift is locally in the direction of the inward normal $n(t,x)$ at any point $(t,x)$ on the parabolic boundary of $\Oc$, and its amplitude is equal to $c_0 |n^T \sigma|(t,x)\sqrt{h}$, with \be\label{c0}
c_0:= \Frac{\E[s^2_{\tau^+}]}{2\E[s_{\tau^+}]}= 0.5826\cdots,
\ee
where $s_0=0, \,\, \forall n\geq 1, s_{n}:=\sum_{i=1}^{n} G^i$, the $G^i$ being i.i.d standard centered normal variables, $\tau^+:=\inf\{n\geq 0:\, s_n >0\}$.
\subsubsection{Numerical scheme for BDSDEs}
\label{Numericalscheme for BDSDEs:subsection}
For each fixed path of $B$, the solution of (\ref{forward})-(\ref{BDSDE}) is approximated by $(Y^{N},Z^{N})$ defined by the following algorithm, given in the multidimensional case.\\
For $0\leq n\leq N-1$:
$\forall j_{1} \in \{1,\ldots,k\}$,
\begin{equation}\label{2}
Y_{t_{n},j_{1}}^{N} = \E_{t_{n}}\Big[Y_{t_{n+1},j_{1}}^{N} + hf_{j_{1}}(X_{t_{n}}^{N},Y_{t_{n}}^{N},Z_{t_{n}}^{N})+\sum_{j=1}^{l}g_{j_{1},j}(X_{t_{n+1}}^{N},Y_{t_{n+1}}^{N},Z_{t_{n+1}}^{N})
\Delta B_{n,j}\Big],
\end{equation}
$\forall j_{1} \in \{1,\ldots,k\} $ and $\forall j_{2} \in \{1,\ldots,d\}$
\begin{equation}\label{3}
h Z_{t_{n},j_{1},j_{2}}^{N} = \E_{t_{n}}\Big[Y_{t_{n+1},j_{1}}^{N}\Delta W_{n,j_{2}}+ \sum_{j=1}^{l}g_{j_{1},j}(X_{t_{n+1}}^{N},Y_{t_{n+1}}^{N},Z_{t_{n+1}}^{N})
\Delta B_{n,j}\Delta W_{n,j_{2}}\Big].
\end{equation}
We stress that at each discretization time, the solution of the algorithm depends on the fixed path of the brownian motion $B$.

\subsubsection{Vector spaces of functions}
At every $t_{n}$, we select $k(d+1)$ deterministic functions bases
$(p_{i,n}(.))_{1\leq i\leq k(d+1)}$ and we look for approximations of $Y_{t_{n}}^{N}$ and $Z_{t_{n}}^{N}$ which will be denoted respectively by $y_{n}^{N}$ and $z_{n}^{N}$, in the vector space spanned by the basis
$(p_{j_{1},n}(.))_{1\leq j_{1}\leq k}$
(respectively $(p_{j_{1},j_{2},n}(.))_{1\leq j_{1}\leq k,1\leq j_{2}\leq d}$).
 Each basis $p_{i,n}(.)$ is considered as a vector of functions of dimension $L_{i,n}$.
In other words, $P_{i,n}(.)=\{ \alpha .p_{i,n}(.), \alpha \in \mathbb{R} ^{L_{i,n}} \}$ where $\alpha$ is the coefficient of the projection on $\L^2(\Omega,\Fc_{t_n})$.\\
As an example, we cite the hypercube basis $(\textbf{HC})$ used in \cite{GLW05}. In this case, $p_{i,n}(.)$ does not depend nor on $i$ neither on $n$ and its dimension is simply denoted by $L$. A domain $D\!\subset\!\mathbb{R}^{d}$ centered on $X_{0}\!=\!x$, that is $D=\!\prod_{i=1}^{d}(x_{i}-a,x_{i}+a]$, can be partitionned on small hypercubes of edge $\delta$. Then, $D\!=\!\bigcup_{i_{1,\ldots,i_{d}}}\!D_{i_{1,\ldots,i_{d}}}$ where $D_{i_{1},\ldots,i_{d}}\!=\!(x_{i}-a+i_{1}\delta,x_{i}-a+i_{1}\delta]\times\ldots\times (x_{i}-a+i_{d}\delta,x_{i}-a+i_{d}\delta]$. Finally we define $p_{i,n}(.)$ as the indicator functions of this set of hypercubes.

\subsubsection{Description of the algorithm}
The main difference with the numerical scheme for FBDSDE in \cite{matouetal13} is the simulation of the first exit time of the forward diffusion process from the domain $\Oc$. The computation of this exit time $\bar{\tau}$ follows a simple and very efficient improved procedure given in \cite{GM10}. The purpose is to stop the Euler scheme at its exit time of a smaller domain in order to compensate the underestimation of exits and to achieve an error of order $o(\sqrt{h})$. The smaller domain is defined by $$\Oc^N:=\{x\in\Oc\, :\, d(x,\partial\Oc) > c_0\sqrt{h}|n^T \sigma(t,x)|\},$$
where  $n(t,x)$ is the inward normal vector at the closest point of  $x$ on the  boundary of $\Oc$ and $c_0$ is the constant given by \eqref{c0}. We shall interpret $|n^T \sigma(t,x)|$ as the noise amplitude along the normal direction to the boundary. Thus the efficient exit time of the Euler scheme is given by
$$\hat{\tau}^N:= \inf \{t_i >0 \,: \, X^N_{t_i}\notin \Oc^n\} \leq \bar{\tau}.$$
For more details on this procedure see the book of Gobet \cite{G13} (page 142-144).

Now the projection coefficients $\alpha$ are computed by using $ M$ independent Monte Carlo simulations
of $X\!_{\!t_{n}\!}\!^{N}$ and $\!\Delta W_{\!n\!}$ which will be respectively denoted by
 $X\!_{\!t_{n}\!}^{N,m}$ and $\Delta \!W_{n}^{m}$, $\!m\!=\!1\!,\ldots,\!M$. The algorithm is explicite as follows:\\[0.5cm]
$\rightarrow$ Initialization: For $n=N$, take $(y_{N}^{N,m})=(\Phi(X^{N,m}_{\hat{\tau}^N}))$ and $(z_{N}^{N,m})=0$ .\\
$\rightarrow$ Iteration: For $n=N-1,\ldots,0$:\\
$\bullet$ We approximate (\ref{3}) by computing for all $j_{1} \in \{1,\ldots,k\}$ and $j_{2} \in \{1,\ldots,d\}$
\begin{eqnarray*}
\alpha^{M}_{j_{1},j_{2},n}&=&\mathop{\rm arginf}\limits_ {\alpha} \frac{1}{M}\sum_{m=1}^{M}\!\Big|y_{n+1,j_{1}}^{N,M,I}(\!X^{N,m}_{t_{n+1}})\!\frac{\!\Delta W_{n,j_{2}}^{m}\!}{h}\\
&\!+\!&\sum_{j=1}^{l}\!g_{j_{1},j}\!\Big(\!X^{N,m}_{t_{n+1}}\!,\!y_{n+1}^{N,M}\!(\!X^{N,m}_{t_{n+1}}\!),z_{n+1}^{N,M}\!(\!X^{N,m}_{t_{n+1}}\!)\!\Big)\!\frac{\Delta B_{n,j}\Delta \!W_{n,j_{2}}^{m}}{h}
-\alpha.p_{j_{1},j_{2},n}^{m}\Big|^{2}.
\end{eqnarray*}
Then we set $z_{n,j_{1},j_{2}}^{N,M}(.)=(\alpha^{M}_{j_{1},j_{2},n}.p_{j_{1},j_{2},n}(.)),\textrm{ }j_{1} \in \{1,\ldots,k\}$, $j_{2} \in \{1,\ldots,d\}$. \\
$\bullet$ We use $I$ Picard iterations to obtain an approximation of $Y_{t_{n}}$ in (\ref{2}):\\
$\cdot$ For $i=0$: $\forall j_{1}\in \{1,\ldots,k\} $, $\alpha^{M,0}_{j_{1},n}=0$.\\
$\cdot$ For $i=1,\ldots,I$: We approximate (\ref{2}) by calculating $\alpha^{M,i}_{j_{1},n}$, $\forall j_{1} \in \{1,\ldots,k\}$, as the minimizer of:
\begin{eqnarray*}
\!\frac{1}{M}\sum_{m=1}^{M}\!\Big|y_{n+1,j_{1}}^{N,M}(X^{N,m}_{t_{n+1}})
\!+hf_{j_{1}}\!\Big(\!X^{N,m}_{t_{n}},\!y_{n}^{N,M,i-1}(X^{N,m}_{t_{n}}\!)\!,\!z_{n}^{N,M}\!(\!X^{N,m}_{t_{n}}\!)\!\Big)\\
+\!\sum_{j=1}^{l}\!g_{j_{1},j}\!\Big(\!X^{N,m}_{t_{n+1}}\!,\!y_{n+1}^{N,M}\!(\!X^{N,m}_{t_{n+1}}\!)\!,\!z^{N,M}_{n+1}\!(\!X^{N,m}_{t_{n+1}}\!)\!\Big)\!\Delta B_{n,j}-\!\alpha .p_{j_{1},k}^{m}\!\Big|^{2}.
\end{eqnarray*}
Finally, we define $y_{n}^{N,M}(.)$ as:
\begin{eqnarray*}
y_{n,j_{1}}^{N,M}(.)=(\alpha^{M}_{j_{1},n}.p_{j_{1},n}(.)),\forall j_{1} \in \{1,\ldots,k\}.
\end{eqnarray*}
\subsection{One-dimensional case (Case when $d=k=l=1$)}
\subsubsection{Function bases}
We use the basis ($\textbf{HC}$) defined above. So we set:
\begin{eqnarray*}
d_{1}=\min_{n,m}X^{m}_{t_{n}},\quad d_{2}=\max_{n,m}X^{m}_{t_{n}} \textrm{ and } L=\frac{d_{2}-d_{1}}{\delta}
\end{eqnarray*}
where $\delta$ is the edge of the hypercubes $(D_{j})_{1\leq j\leq L}$ defined by $D_{j}=\Big[d+(j-1)\delta,d+j\delta\Big),\forall j$.\\
We take at each time $t_{n}$
\begin{eqnarray*}
1_{D_{j}}(X^{N,m}_{t_{n}})=1_{[d+(j-1)\delta,d+j\delta)} (X^{N,m}_{t_{n}}),j=1,\ldots,L
\end{eqnarray*}
and
\begin{eqnarray*}
\!(\!p^{m}_{i,n}(.)\!)\!=\!\Big\{\!\sqrt{\frac{M}{card(D_{j})}}\!1_{D_{j}}\!(\!X^{N,m}_{t_{n}}\!)\!,\!1\! \leq \!j\!\leq\! L\Big\},i=0,1.
\end{eqnarray*}
$Card(D_{j})$ denotes the number of simulations of $X^{N}_{t_{n}}$ which are in our cube $D_{j}$.\\
This system is orthonormal with respect to the empirical scalar product defined by
\begin{eqnarray*}
<\psi_{1},\psi_{2}>_{n,M}:=\frac{1}{M}\!\sum_{m=1}^{M}\!\psi_{1}\!(\!X^{N,m}_{t_{n}}\!) \psi_{2}\!(X^{N,m}_{t_{n}}\!).
\end{eqnarray*}
In this case, the solutions of our least squares problems are given by:
\begin{eqnarray*}
\alpha^{M}_{1,n}&=&\frac{1}{M}\sum_{m=1}^{M} p_{1,n}(X^{N,m}_{t_{n}}) \Big\{ y_{n+1}^{N,M}(X^{N,m}_{t_{n+1}}) \frac{\Delta W^{m}_{n}}{h}\\
&+&g\Big(X^{N,m}_{t_{n+1}},y_{n+1}^{N,M}(X^{N,m}_{t_{n+1}}),z^{N,M,}_{n+1}(X^{N,m}_{t_{n+1}})\Big)\frac{\Delta B^{m}_{n}\Delta W^{m}_{n}}{h}\Big\},\\
\alpha^{M,i}_{0,n}&=&\frac{1}{M}\sum_{m=1}^{M} p_{0,n}(X^{N,m}_{t_{n}}) \Big\{y_{n+1}^{N,M}(X^{N,m}_{t_{n+1}})+ h f\Big(X_{t_{n}}^{N,m},y_{n}^{N,M,i-1}(X^{N,m}_{t_{n}}),z_{n}^{N,M}(X^{N,m}_{t_{n}})\Big)\\
&+&g\Big(X^{N,m}_{t_{n+1}},y_{n+1}^{N,M}(X^{N,m}_{t_{n+1}}),z^{N,M}_{n+1}(X^{N,m}_{t_{n+1}})\Big)\Delta B^{m}_{n}\Big\}.
\end{eqnarray*}
\begin{Remark}
We note that for each value of $M$, $N$ and $\delta$, we launch the algorithm $50$ times and we denote by $(Y_{0,m'}^{0,x,N,M})_{1\leq m' \leq50}$ the set of collected values. Then we calculate the empirical mean $\overline{Y}_{0}^{0,x,N,M,I}$ and the empirical standard deviation $\sigma^{N,M}$defined by:
\begin{equation}\label{approxY0}
\overline{Y}_{0}^{0,x,N,M}\!=\frac{1}{50}\sum_{m'=1}^{50}\!Y_{0,m'}^{0,x,N,M} \textrm{ and } \sigma^{N,M}\!=\!\sqrt{\frac{1}{49}\sum_{m'=1}^{50}|Y_{0,m'}^{0,x,N,M}\!-\!\overline{Y}_{0}^{0,x,N,M}|^{2}}.
\end{equation}
We also note before starting the numerical examples that our algorithm converges after at most three Picard iterations.
Finally, we stress that (\ref{approxY0}) gives us an approximation of $u(0,x)$ the solution of the SPDE (\ref{SPDEdirichlet}) at time $t=0$.
\end{Remark}

\subsubsection{Comparison of numerical approximations of the solutions of the FBDSDE and the FBSDE: the general case}

Now we take
\begin{eqnarray*}
\begin{cases}
&\Phi(x)=-x+K,\\
&f(t,x,y,z)=-\theta z-ry+(y-\frac{z}{\sigma})^{-}(R-r),\\
&g_1(t,x,y,z)=0.1z+0.5y+log(x)
\end{cases}
\end{eqnarray*}
and we set $\theta=(\mu-r)/\sigma$, $K=115$, $X_{0}=100$, $\mu=0.05$, $\sigma=0.2$, $r=0.01$, $R=0.06$, $\delta=1$, $N=20$, $T=0.25$ and we fix $d_{1}=60$ and $d_{2}=200$ as in \cite{GLW06}. We fix our domain $\Oc=]60,200[$, choosen large enough to compenstate the rate of convergence of the exit time approximation which is slow (of order $h^{1/2}$).\\
The functions $g_1$,$g_2$ and $g_3$ taken in what follows are examples of the function g. They are sufficiently regular and Lipschitz on $[60,200]\times\mathbb{R}\times\mathbb{R}$ and could be extended to regular Lipschitz functions on $\mathbb{R}^3$. In this case, the continuous Lipschitz assumption is satisfied.\\
We compare the numerical solution of our BDSDE with terminal time $\bar{\tau}$ (noted again $\overline{Y}_{t}^{t,x,N,M}$), the BDSDE's one (noted here by $\overline{Y}_{t,BDSDE}^{0,x,N,M}$ ) and the BSDE's one (noted here by $\overline{Y}_{t,BSDE}^{0,x,N,M}$ ), without $g$ and $B$. Note that each CPU-Time given in the tables is for 50 macro-runs of the algorithms.\\

When $t$ is close to maturity $t=t_{19}$\\

\begin{center}
\begin{tabular}{r c c c c}

\hline
$M$ & $\overline{Y}_{t_{19},BSDE}^{0,x,N,M}(\sigma^{N,M})$ &$\overline{Y}_{t_{19},BDSDE}^{0,x,N,M}(\sigma^{N,M})$& $\overline{Y}_{t_{19}}^{0,x,N,M}(\sigma^{N,M})$ \\

\hline\
  128& 13.748(0.879)& 15.453(0.948) & 13.392(1.021)\\

\hline
  512& 13.827(0.384)& 15.535(0.409) & 12.210(0.3580)\\

\hline

  2048& 13.762(0.223)& 15.465(0.240)& 12.051(0.197)\\

\hline

  8192& 13.781(0.091)& 15.485(0.097)& 14.814 (0.107)\\

\hline

  32768& 13.796(0.054)& 15.501(0.058)& 14.729 (0.053)\\

\hline

\end{tabular}
\vspace{0.5cm}
\end{center}

when $t=t_{15}$\\
\begin{center}
\begin{tabular}{r c c c }

\hline
$M$ & $\overline{Y}_{t_{15},BSDE}^{0,x,N,M}(\sigma^{N,M})$ &$\overline{Y}_{t_{15},BDSDE}^{0,x,N,M}(\sigma^{N,M})$& $\overline{Y}_{t_{15}}^{0,x,N,M}(\sigma^{N,M})$ \\

\hline\
  128& 14.168(0.905)& 17.894(1.096) & 13.049(1.116)\\

\hline
  512& 14.113(0.388)& 17.774(0.429)& 16.469(0.441) \\

\hline

  2048& 13.988(0.226)& 17.607(0.270) & 9.817(0.178)\\

\hline

  8192& 13.985(0.093)& 17.623(0.104)& 12.951(0.115)\\

\hline

  32768& 13.994(0.055)& 17.627(0.064)& 13.232(0.053)\\

\hline

\end{tabular}
\vspace{0.5cm}
\end{center}

when $t=0$\\
\begin{center}
\begin{tabular}{ r c c c c}

\hline
$M$ & $\overline{Y}_{0,BSDE}^{0,x,N,M}(\sigma^{N,M})$ & $\overline{Y}_{0,BDSDE}^{0,x,N,M}(\sigma^{N,M})$  & $\overline{Y}_{0}^{0,x,N,M}(\sigma^{N,M})$ & CPU-Time (sec)\\

\hline\
  128& 15.431(1.005)& 13.571(1.146)& 19.719(1.558)& 2.418
 \\

\hline
  512& 15.029(0.428)& 13.173(0.500)& 24.371(0.659)& 10.234 \\

\hline

  2048& 14.763(0.243)& 12.885(0.280) & 13.433(0.233) & 46.882\\

\hline

  8192& 14.718(0.098)& 12.825(0.106)& 12.543(0.122)& 220.531\\

\hline

  32768& 14.715(0.060)& 12.804(0.064)& 13.458(0.057)& 940.531\\

\hline

\end{tabular}
\vspace{0.5cm}
\end{center}

In the previous tables, we test our algorithm for different times (when they are close to the maturity and in initial time $t=0$) and we modify the number of Monte Carlo simulation $M$ for fixed number of time discretization $N$. We note that  
the numerical value of the BDSDE with random terminal time $\bar{\tau}$ converges to the value of classical BDSDE for $M$ large and this can be explained by the fact that the approximated value of the exit time is close to the maturity $T$.\\[0.5cm]
For $g_{2}(y,z)=0.1z+0.5y$ when $t=t_{19}$.

\begin{center}
\vspace{0.2cm}
\begin{tabular}{r c c c}

\hline
$M$ & $\overline{Y}_{t_{19},BDSDE}^{0,x,N,M}(\sigma^{N,M})$&$\overline{Y}_{t_{19}}^{0,x,N,M,I}(\sigma^{N,M})$ \\

\hline\
  128& 14.767(0.949)& 13.545(1.020)
 \\

\hline
  512& 14.850(0.410)& 12.862(0.358) \\

\hline

  2048& 14.781(0.240)& 12.739(0.197)
 \\

\hline

  8192& 14.801(0.097)& 14.401(0.107)
\\

\hline

  32768& 14.818(0.058)& 14.358(0.053)
\\

\hline

\end{tabular}
\vspace{0.5cm}
\end{center}
when $t=t_{15}$
\begin{center}
\begin{tabular}{r c c c}

\hline
$M$ & $\overline{Y}_{t_{15},BDSDE}^{0,x,N,M}(\sigma^{N,M})$&$\overline{Y}_{t_{15}}^{0,x,N,,M}(\sigma^{N,M})$ \\

\hline\
  128& 16.267(1.093)& 13.607(1.111)
 \\

\hline
  512& 16.166(0.428)& 15.191(0.443)
\\

\hline

  2048& 16.007(0.270)& 11.675(0.180)
\\

\hline

  8192& 16.024(0.104) & 13.551(0.114)\\

\hline

  32768& 16.029(0.064)& 13.689(0.053)\\

\hline

\end{tabular}
\vspace{1cm}
\end{center}

when $t=0$
\begin{center}
\begin{tabular}{r c c c c}

\hline
$M$ &$\overline{Y}_{0,BDSDE}^{0,x,N,M}(\sigma^{N,M})$&  $\overline{Y}_{0}^{0,x,N,M}(\sigma^{N,M})$ & CPU-Time (sec)\\

\hline\
  128& 13.821(0.063)& 17.811(1.529)& 2.401 \\

\hline
  512& 14.555(1.132)& 19.766(0.645) & 10.182 \\

\hline

  2048& 14.176(0.495)& 13.976(0.241)& 46.679\\

\hline

  8192& 13.899(0.277)& 13.635(0.122)& 223.694\\

\hline

  32768& 13.842(0.105)& 14.139(0.058)& 969.030\\

\hline

\end{tabular}
\vspace{0.5cm}
\end{center}

For $g_{3}(x,y)=logx +0.5y$:
when $t$ is close to maturity $t=t_{19}$.
\begin{center}
\vspace{0.5cm}
\begin{tabular}{r c c c}

\hline
$M$ & $\overline{Y}_{t_{19},BDSDE}^{0,x,N,M}(\sigma^{N,M})$&$\overline{Y}_{t_{19}}^{0,x,N,M}(\sigma^{N,M})$ \\

\hline\
  128& 15.452(0.948)& 13.392(1.021) \\

\hline
  512& 15.534(0.409)& 12.210(0.358) \\

\hline

  2048& 15.464(0.240)& 12.051(0.197)\\

\hline

  8192& 15.484(0.097)& 14.814(0.107)
\\

\hline

  32768& 15.501(0.058)& 14.729(0.053)\\

\hline

\end{tabular}
\vspace{0.5cm}
\end{center}
when $t=t_{15}$
\begin{center}
\begin{tabular}{r c c c}

\hline
$M$ & $\overline{Y}_{t_{15},BDSDE}^{0,x,N,M}(\sigma^{N,M})$&$\overline{Y}_{t_{15}}^{0,x,N,M}(\sigma^{N,M})$ \\
\hline\
  128& 18.253(1.068)& 12.782(1.003)

 \\

\hline
  512& 18.166(0.453)& 17.383(0.454)\\

\hline

  2048& 18.010(0.266)& 9.325(0.174)\\

\hline

  8192& 18.006(0.109)& 12.490(0.097)\\

\hline

  32768& 18.017(0.065)& 12.858(0.049)

\\

\hline

\end{tabular}
\vspace{0.5cm}
\end{center}

when $t=0$
\begin{center}
\begin{tabular}{r c c c c}

\hline
$M$ &$\overline{Y}_{0,BDSDE}^{0,x,N,M}(\sigma^{N,M})$&  $\overline{Y}_{0}^{0,x,N,M,}(\sigma^{N,M})$ & CPU-Time (sec)\\

\hline\
  128& 12.071(0.054)& 20.496(1.421) & 2.401\\

\hline
  512& 12.075(0.088)& 27.093(0.654) & 10.322\\

\hline

  2048& 12.122(0.218)& 13.362(0.221)& 47.039\\

\hline

  8192& 12.384(0.381)& 11.878(0.101)& 221.504\\

\hline

  32768& 12.791(0.903)& 12.948(0.051)& 938.669\\

\hline

\end{tabular}
\vspace{0.5cm}
\end{center}
In the previous tables, we test our algorithm for different examples of the function $g$ ($g_1$ and $g_2$ are dependent in $z$, $g_3$ is independent of $z$).

\begin{figure}[h!]
\includegraphics[width=0.75\textwidth]{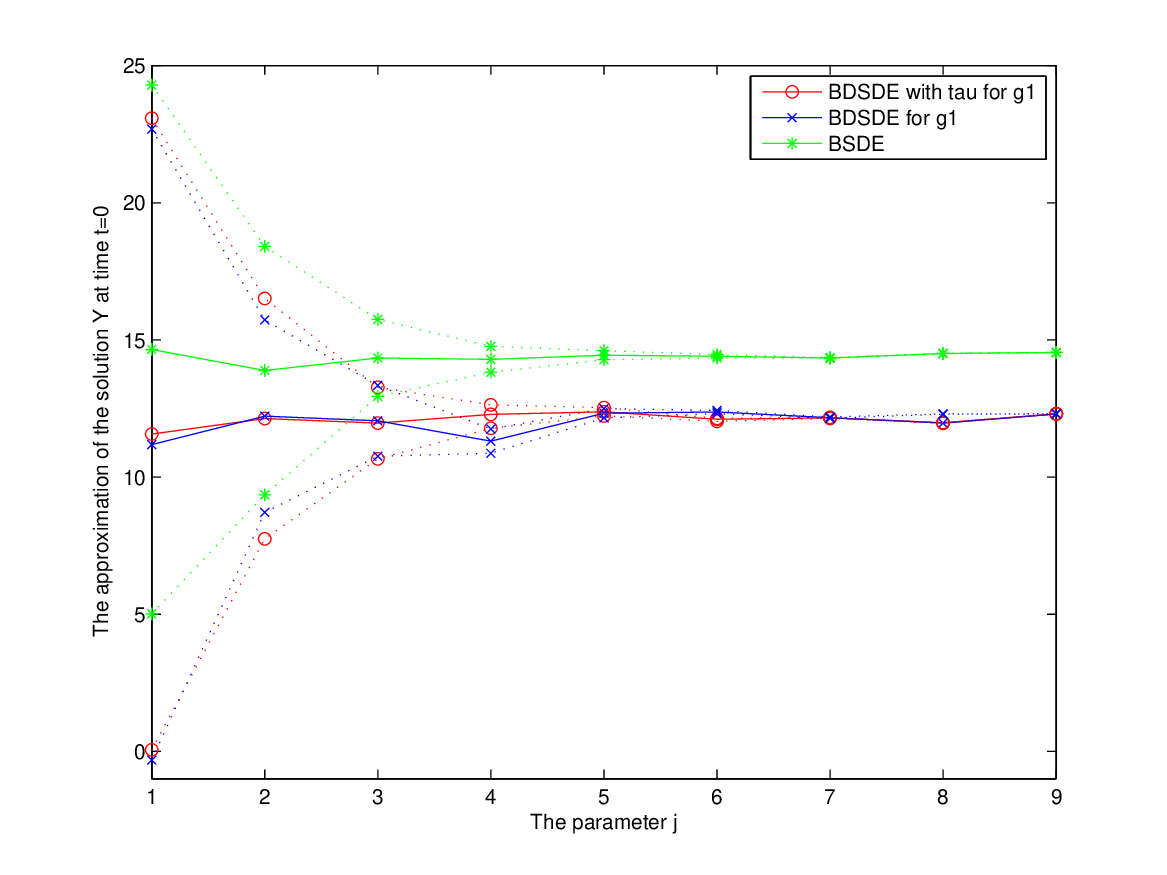} 
\caption{\label{figure1}Comparisom of the BSDE's solution, the BDSDE's one and the solution of BDSDE with random time for $g_{1}(x,y,z)= log(x)+ 0.5 y+ 0.1 z$. Confidence intervals are with dotted lines. }
\end{figure}

We see on Figure \ref{figure1} and \ref{figure2} the impact of the function $g$ on the solution; we modify $N$, $M$ and $\delta$ as in \cite{LGW06}, by taking these quantities as follows: First we fix $d_{1}=40$ and $d_{2}=180$ (which means that  $x \in [d_{1},d_{2}]=[40,180]$ and in this case our continuous lipschitz assumptions are satisfied). Let $j \in \mathbb{N}$, we take $\alpha_{M}=3$, $\beta=1$, $N=2(\sqrt{2})^{(j-1)}$, $M=2(\sqrt{2})^{\alpha_{M}(j-1)}$ and $\delta=50/(\sqrt{2})^{(j-1)(\beta+1)/2}$. Then, we draw the map of each solution at $t=0$ with respect to j. We remark from the figures that numerical values of the BDSDE with random terminal time coincide with that of the clasical BDSDE after just few variation of the parameters. This allow us to think about performing the rate of convergence of our algorithm by getting weaker estimates for the BDSDE (as Bouchard and Menozzi for the classical BSDEs \cite{BM09}).

\begin{figure}[h!]
\centering
\includegraphics[width=0.75\textwidth]{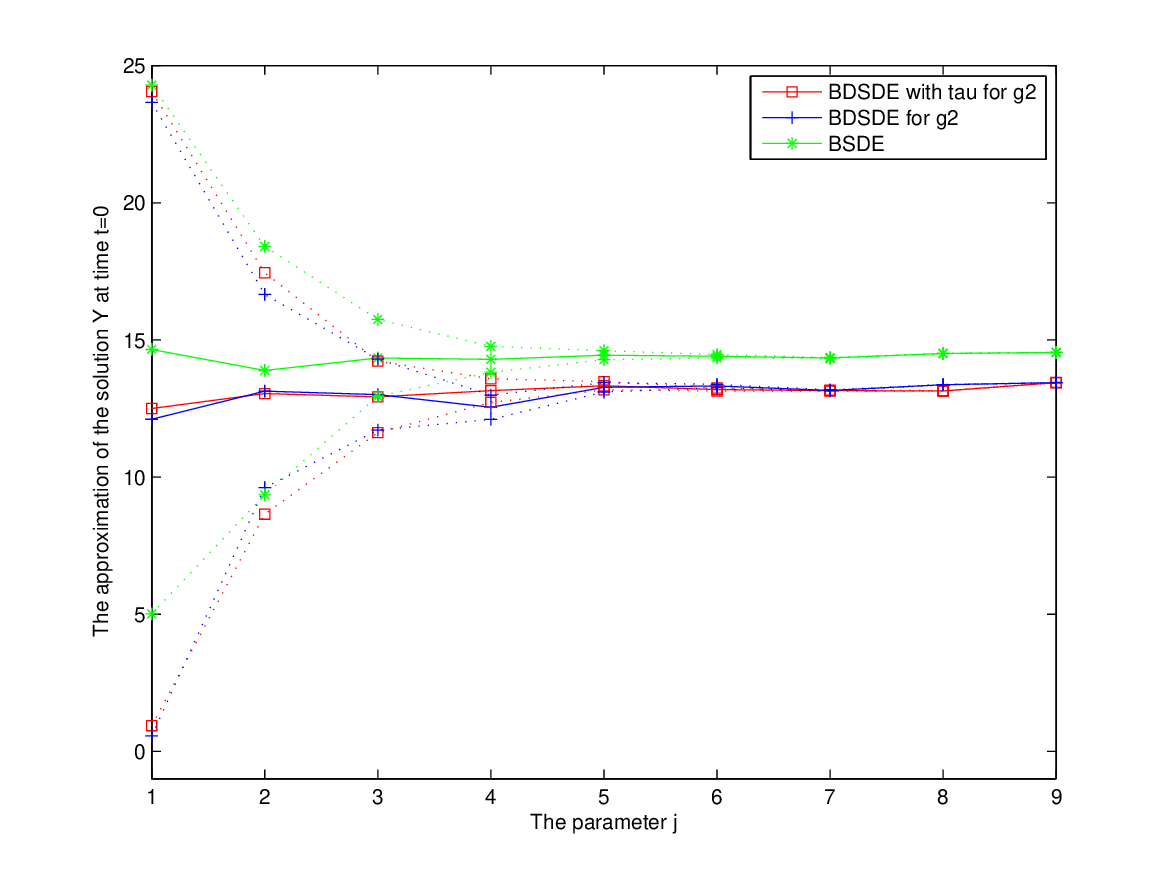} 
\caption{\label{figure2}Comparisom of the BSDE's solution, the BDSDE's one and the solution of BDSDE with random time for $g_{2}(x,y,z)= 0.5 y+ 0.1 z$. Confidence intervals are with dotted lines.}
\end{figure}

\section{Conclusion}

The main result of this paper is to develop  a discrete-time approximation of a Forward-Backward Doubly SDE with finite stopping time horizon, namely the first exit time of a forward SDE from a  domain $\Oc$. More precisely we provide a rate of convergence of order $ h^{1/2}$  for the square of Euler time discretization error  for the scheme \eqref{euldiscrete}-\eqref{Yn} (Theorem \ref{finalerror}). This order is performed compared to the one obtained by Bouchard and Menozzi \cite{BM14} in the case of BSDE with random terminal time (the strong error) thanks to the recent work of Bouchard, Geiss and Gobet  \cite{BGG13}. Moreover, Euler scheme for a class of semilinear SPDEs with Cauchy-Dirichlet condition is provided via the scheme of the Forward-Backward Doubly SDE \eqref{euldiscrete}-\eqref{Yn}, this gives a probabilistic point of view for the approximation error of this class of SPDEs.
\paragraph*{Aknowledgment} The authors whish to thank the associate editor and the anonymous referee for the pertinent remark she/he made.
\bibliographystyle{acm}
\bibliography{NumBDSDE}

\begin{thebibliography}{10}

\bibitem{A09}
{\sc Aboura, O.}
\newblock On the discretization of backward doubly stochastic differential
  equations.
\newblock {\em arXiv:0907.1406\/} (2009).

\bibitem{A13}
{\sc Aman, A.}
\newblock A numerical scheme for backward doubly stochastic differential
  equations.
\newblock {\em Bernoulli 19}, 1 (2013), 93--114.

\bibitem{B14}
{\sc Bachouch, A.}
\newblock Numerical computations for backward doubly stochastic differential
  equations and non-linear stochastic pdes.
\newblock {\em Universit\'e du Maine, PhD Thesis\/} (2014).

\bibitem{BM14}
{\sc Bachouch, A., Gobet, E., and Matoussi, A.}
\newblock Numerical computation for quasilinear {SPDE}s via generalized
  backward doubly {SDE}s.
\newblock {\em Forthcoming paper\/}.

\bibitem{BGM14a}
{\sc Bachouch, A., Gobet, E., and Matoussi, A.}
\newblock Empirical {R}egression {M}ethod for {B}ackward {D}oubly {S}tochastic
  {D}ifferential {E}quations.
\newblock {\em SIAM/ASA J. Uncertain. Quantif. 4}, 1 (2016), 358--379.

\bibitem{matouetal13}
{\sc Bachouch, A., Lasmar, A.~B., Matoussi, A., and Mnif, M.}
\newblock Numerical scheme for semilinear {SPDE}s via backward doubly {SDE}s.
\newblock {\em Stochastic Partial Differential Equations: Analysis and
  Computation. 1\/} (2016), 1--43.

\bibitem{B98}
{\sc Bally, V.}
\newblock Construction of asymptotically optimal controls for control and game
  problems.
\newblock {\em Probab. Theory Related Fields 111}, 3 (1998), 453--467.

\bibitem{BM01}
{\sc Bally, V., and Matoussi, A.}
\newblock Weak solutions for {SPDE}s and backward doubly stochastic
  differential equations.
\newblock {\em J. Theoret. Probab. 14}, 1 (2001), 125--164.

\bibitem{BPS05}
{\sc Bally, V., Pardoux, E., and Stoica, L.}
\newblock Backward stochastic differential equations associated to a symmetric
  {M}arkov process.
\newblock {\em Potential Anal. 22}, 1 (2005), 17--60.

\bibitem{BL97}
{\sc Barles, G., and Lesigne, E.}
\newblock S{DE}, {BSDE} and {PDE}.
\newblock In {\em Backward stochastic differential equations ({P}aris,
  1995--1996)}, vol.~364 of {\em Pitman Res. Notes Math. Ser.} Longman, Harlow,
  1997, pp.~47--80.

\bibitem{BGG13}
{\sc Bouchard, B., Geiss, S., and Gobet, E.}
\newblock First time to exit of a continuous it\^o process: general moment
  estimates and $l_1$-convergence rate for discrete time approximations.
\newblock {\em arXiv:1307.4247\/}.

\bibitem{BM09}
{\sc Bouchard, B., and Menozzi, S.}
\newblock Strong approximations of {BSDE}s in a domain.
\newblock {\em Bernoulli 15}, 4 (2009), 1117--1147.

\bibitem{BT04}
{\sc Bouchard, B., and Touzi, N.}
\newblock Discrete-time approximation and {M}onte-{C}arlo simulation of
  backward stochastic differential equations.
\newblock {\em Stochastic Process. Appl. 111}, 2 (2004), 175--206.

\bibitem{BDHPS03}
{\sc Briand, P., Delyon, B., Hu, Y., Pardoux, E., and Stoica, L.}
\newblock {$L^p$} solutions of backward stochastic differential equations.
\newblock {\em Stochastic Process. Appl. 108}, 1 (2003), 109--129.

\bibitem{buck:ma:10b}
{\sc Buckdahn, R., and Ma, J.}
\newblock Stochastic viscosity solutions for nonlinear stochastic partial
  differential equations. {I}.
\newblock {\em Stochastic Process. Appl. 93}, 2 (2001), 181--204.

\bibitem{buck:ma:10a}
{\sc Buckdahn, R., and Ma, J.}
\newblock Stochastic viscosity solutions for nonlinear stochastic partial
  differential equations. {II}.
\newblock {\em Stochastic Process. Appl. 93}, 2 (2001), 205--228.

\bibitem{CMT10}
{\sc Crisan, D., Manolarakis, K., and Touzi, N.}
\newblock On the {M}onte {C}arlo simulation of {BSDE}s: an improvement on the
  {M}alliavin weights.
\newblock {\em Stochastic Process. Appl. 120}, 7 (2010), 1133--1158.

\bibitem{DP97}
{\sc Darling, R. W.~R., and Pardoux, E.}
\newblock Backwards {SDE} with random terminal time and applications to
  semilinear elliptic {PDE}.
\newblock {\em Ann. Probab. 25}, 3 (1997), 1135--1159.

\bibitem{DMS09}
{\sc Denis, L., Matoussi, A., and Stoica, L.}
\newblock Maximum principle and comparison theorems for solutions of
  quasilinear {SPDE}'s.
\newblock {\em Electronic Journal of Probability 14\/} (2009), 500--530.

\bibitem{DMS10}
{\sc Denis, L., Matoussi, A., and Stoica, L.}
\newblock Moser iteration applied to parabolic {SPDE}'s: first approach.
\newblock In {\em Stochastic partial differential equations and applications},
  vol.~25 of {\em Quad. Mat.} Dept. Math., Seconda Univ. Napoli, Caserta, 2010,
  pp.~99--125.

\bibitem{DS04}
{\sc Denis, L., and Stoica, L.}
\newblock A general analytical result for non-linear {SPDE}'s and applications.
\newblock {\em Electron. J. Probab. 9\/} (2004), no. 23, 674--709 (electronic).

\bibitem{Gob98}
{\sc Gobet, E.}
\newblock Sch\'ema d'{E}uler continu pour des diffusions tu\'ees et options
  barri\`ere.
\newblock {\em C. R. Acad. Sci. Paris S\'er. I Math. 326}, 12 (1998),
  1411--1414.

\bibitem{Gob2000}
{\sc Gobet, E.}
\newblock Weak approximation of killed diffusion using {E}uler schemes.
\newblock {\em Stochastic Process. Appl. 87}, 2 (2000), 167--197.

\bibitem{G13}
{\sc Gobet, E.}
\newblock {\em M\'ethodes de Monte-Carlo et processus stochastiques: du
  lin\'eaire au non-lin\'eaire}.
\newblock \'Ecole Polytechnique, 2013.

\bibitem{GLW05}
{\sc Gobet, E., Lemor, J.-P., and Warin, X.}
\newblock A regression-based {M}onte {C}arlo method to solve backward
  stochastic differential equations.
\newblock {\em Ann. Appl. Probab. 15}, 3 (2005), 2172--2202.

\bibitem{GLW06}
{\sc Gobet, E., Lemor, J.-P., and Warin, X.}
\newblock Rate of convergence of an empirical regression method for solving
  generalized backward stochastic differential equations.
\newblock {\em Bernoulli 12}, 5 (2006), 889--916.

\bibitem{GM10}
{\sc Gobet, E., and Menozzi, S.}
\newblock Stopped diffusion processes: boundary corrections and overshoot.
\newblock {\em Stochastic Process. Appl. 120}, 2 (2010), 130--162.

\bibitem{GK10}
{\sc Gy{\"o}ngy, I., and Krylov, N.}
\newblock Accelerated finite difference schemes for linear stochastic partial
  differential equations in the whole space.
\newblock {\em SIAM J. Math. Anal. 42}, 5 (2010), 2275--2296.

\bibitem{JK10}
{\sc Jentzen, A., and Kloeden, P.}
\newblock Taylor expansions of solutions of stochastic partial differential
  equations with additive noise.
\newblock {\em Ann. Probab. 38}, 2 (2010), 532--569.

\bibitem{K84}
{\sc Kunita, H.}
\newblock Stochastic differential equations and stochastic flows of
  diffeomorphisms.
\newblock In {\em \'{E}cole d'\'et\'e de probabilit\'es de {S}aint-{F}lour,
  {XII}---1982}, vol.~1097 of {\em Lecture Notes in Math.} Springer, Berlin,
  1984, pp.~143--303.

\bibitem{K94b}
{\sc Kunita, H.}
\newblock Generalized solutions of a stochastic partial differential equation.
\newblock {\em J. Theoret. Probab. 7}, 2 (1994), 279--308.

\bibitem{K94a}
{\sc Kunita, H.}
\newblock Stochastic flow acting on schwartz distributions.
\newblock {\em Journal of Theoretical Probability 7}, 2 (1994), 279--308.

\bibitem{LGW06}
{\sc Lemor, J.-P., Gobet, E., and Warin, X.}
\newblock Rate of convergence of an empirical regression method for solving
  generalized backward stochastic differential equations.
\newblock {\em Bernoulli 12}, 5 (2006), 889--916.

\bibitem{lion:soug:98}
{\sc Lions, P.-L., and Souganidis, P.~E.}
\newblock Fully nonlinear viscosity stochastic partial differential equations:
  non-smooth equations and applications.
\newblock {\em C.R. Acad. Sci. Paris 327}, 1 (1998), 735--741.

\bibitem{lion:soug:00}
{\sc Lions, P.-L., and Souganidis, P.~E.}
\newblock \'{E}quations aux d\'eriv\'ees partielles stochastiques
  nonlin\'eaires et solutions de viscosit\'e.
\newblock Exp. No. I, 15.

\bibitem{lion:soug:01}
{\sc Lions, P.-L., and Souganidis, P.~E.}
\newblock Viscosity solutions of fully nonlinear stochastic partial
  differential equations.
\newblock {\em S\=urikaisekikenky\=usho K\=oky\=uroku}, 1287 (2002), 58--65.

\bibitem{MPP15}
{\sc Matoussi, A., Popier, A., and Piozin, L.}
\newblock Stochastic partial differential equations with singular terminal
  condition.
\newblock {\em arXiv:1412.5548\/} (to appear in Stochastic Processes and
  Applications).

\bibitem{MS02}
{\sc Matoussi, A., and Scheutzow, M.}
\newblock Semilinear stochastic pde's with nonlinear noise and backward doubly
  sde's.
\newblock {\em Journal of Theoretical Probability 15\/} (2002), 1--39.

\bibitem{pp1994}
{\sc Pardoux, E., and Peng, S.}
\newblock Backward doubly sde's and systems of quasilinear spdes.
\newblock {\em Probab. Theory and Related Field 98\/} (1994), 209--227.

\bibitem{P91}
{\sc Peng, S.~G.}
\newblock Probabilistic interpretation for systems of quasilinear parabolic
  partial differential equations.
\newblock {\em Stochastics Stochastics Rep. 37}, 1-2 (1991), 61--74.

\bibitem{SV72}
{\sc Stroock, D., and Varadhan, S. R.~S.}
\newblock On degenerate elliptic-parabolic operators of second order and their
  associated diffusions.
\newblock {\em Comm. Pure Appl. Math. 25\/} (1972), 651--713.

\bibitem{W05}
{\sc Walsh, J.~B.}
\newblock Finite element methods for parabolic stochastic {PDE}'s.
\newblock {\em Potential Anal. 23}, 1 (2005), 1--43.

\bibitem{Z04}
{\sc Zhang, J.}
\newblock A numerical scheme for {BSDEs}.
\newblock {\em Ann. Appl. Probab. 14}, 1 (2004), 459--488.

\end{thebibliography}


\end{document}